\documentclass[]{article}

\newcommand{\un}{\text{ $\underline{\triangleright}$ }}
\newcommand{\ovr}{\text{ $\overline{\triangleright}$ }}
\newcommand{\Hom}{\operatorname{Hom}}
\newcommand{\Col}{\operatorname{Col}}
\newcommand{\Z}{\mathbb{Z}}
\newcommand{\bvb}[1]{\langle #1\rangle_{\beta}}
\newcommand{\bvbn}[1]{\langle #1\rangle_{\beta}^{\operatorname{n}}}
\newcommand{\fbvb}[1]{\langle #1\rangle_{\mathfrak{B}}}
\newcommand{\fbvbn}[1]{\langle #1\rangle_{\mathfrak{B}}^{\operatorname{n}}}
\newcommand{\bvbm}[1]{\mathcal{M}_X^{\beta}(#1)}
\newcommand{\bcm}[1]{\mathcal{M}_X(#1)}
\newcommand{\bvbmn}[1]{\widetilde{\mathcal{M}}_X^{\beta}(#1)}
\newcommand{\bvbmult}[1]{\Phi_X^{\beta,M}(#1)}
\newcommand{\bvbnmult}[1]{\widetilde{\Phi}_X^{\beta,M}(#1)}

\usepackage[a4paper,margin=1in]{geometry}
\usepackage[none]{hyphenat}
\usepackage{xcolor, array, amssymb}
\usepackage{amsmath,amsfonts,graphicx,amsthm,setspace,subcaption,overpic}
\usepackage{etoolbox,tikz-cd,rotating,url,booktabs}

\theoremstyle{definition}
\newtheorem{theorem}{Theorem}[subsection]
\newtheorem{example}{Example}[subsection]
\newtheorem{definition}{Definition}[subsection]
\newtheorem{proposition}{Proposition}[subsection]

\newtheorem{corollary}{Corollary}[subsection]
\newtheorem{remark}{Remark}[subsection]

\author{
	Hamdi Kayaslan\thanks{Dept. of Mathematics, Izmir Institute of Technology, Izmir, Turkey. Email: \texttt{hamdikayaslan@iyte.edu.tr}.}
	\and
	Sel\c{c}uk \.Ilbeyli\thanks{Dept. of Mathematics, Izmir Institute of Technology, Izmir, Turkey. Email: \texttt{selcukilbeyli@std.iyte.edu.tr}.}
}

\title{Biquandle-Based Invariants of Virtual Knotoids under Connected Sum}
\date{}

\begin{document}
	\maketitle
    \onehalfspacing
	\begin{abstract}
     We study the behavior of biquandle-based invariants of virtual knotoids under their connected sum. We show that the fundamental biquandle of the connected sum of two virtual knotoids is the pushout of a span in the category of biquandles. In particular,
        \begin{align}
        \label{fundamental-formula}
        \mathcal B(K_1\#K_2)
        \cong
        \mathcal B(K_1)*_{F_1}\mathcal B(K_2),
        \end{align}
        where the right-hand side is the amalgamated product of the fundamental biquandles of the summands over the free biquandle on one generator $F_1$. By applying the Hom functor to \eqref{fundamental-formula}, we obtain the correspondence between biquandle colorings of $K_1\# K_2$ and compatible pairs of colorings of summands. This provides a categorical explanation of a known matrix product formula for biquandle counting matrices under connected sum.
        
        Next, we introduce normalized multiset invariants derived from biquandle virtual brackets and show that they are strictly stronger than their unnormalized versions in certain cases by an explicit example. We then study the behavior of these invariants under connected sum. We show that, for each coloring of the connected sum $K_1\#K_2$ corresponding to a compatible pair of colorings of the summands $K_1$ and $K_2$, the normalized biquandle virtual bracket value factors as the product of the values of the summands. That is,
        \[
        \langle (K_1\#K_2)_{f} \rangle_{\beta}^{\operatorname{n}}=\langle (K_1)_{f_1} \rangle_{\beta}^{\operatorname{n}}\langle (K_2)_{f_2} \rangle_{\beta}^{\operatorname{n}}.
        \]
        Building on this, we obtain connected-sum formulas for the normalized multiset invariants defined by utilizing biquandle virtual brackets.
        
        When the coefficient ring is a number ring, the normalized bracket multisets can be encoded by polynomials and polynomial matrices. We introduce a polynomial matrix product $\odot$ and show that the normalized biquandle virtual bracket matrices satisfy
        \[
        \widetilde{\mathcal{M}}_X^{\beta}(K_1\#K_2)
        =
        \widetilde{\mathcal{M}}_X^{\beta}(K_1)
        \odot
        \widetilde{\mathcal{M}}_X^{\beta}(K_2).
        \]
	\end{abstract}

\section{Introduction}

    Knotoids were introduced by Turaev in~\cite{turaev2012knotoids} as the
    theory of open-ended knot diagrams in oriented surfaces. A knotoid diagram is based
    on an oriented interval with two distinguished endpoints, called the
    \textit{tail} and the \textit{head}. Knotoid diagrams are considered up to
    \textit{Reidemeister moves} performed away from the endpoints, while a passage of an
    endpoint over or under another strand is forbidden. The theory of knotoids has been shown to be useful in applications to open chains, including protein chains and other polymeric structures \cite{goundaroulis2017studies,goundaroulis2017topological,dorier2018tool}.

    Virtual knot theory, introduced by Kauffman in~\cite{kauffman1999virtual},
    extends diagrammatic classical knot theory by introducing \textit{virtual crossings} and adding the \textit{detour move} to the usual Reidemeister moves in the equivalence relation. This theory also admits a geometric
    interpretation in terms of knots in thickened surfaces, considered up to
    stabilization and destabilization \cite{carter2002stable,kamada2000abstract}. 
    
    Virtual knotoids combine
    the endpoint structure of knotoids with the virtual framework of the virtual knot theory. The theory was
    introduced by Gügümcü and Kauffman in~\cite{gugumcu2017new}, together with a surface interpretation of the theory in terms of knotoids in surfaces. A three-dimensional interpretation in terms of rail arcs in thickened surfaces was given by Gügümcü and the first author in
    \cite{gugumcu2025virtual}.

    Biquandles provide an algebraic framework for coloring classical and virtual knot and knotoid diagrams. Biquandle operations determine the local coloring rules at classical crossings. An example of a biquandle is the \textit{fundamental biquandle} $\mathcal{B}(K)$ of a virtual knotoid diagram $K$. The fundamental biquandle is presented by generators corresponding to the semiarcs of $K$, up to the relations determined by the classical crossings of $K$. For a given biquandle $X$, an \textit{$X$-coloring} of $K$ is a biquandle homomorphism $\mathcal{B}(K)\rightarrow X$. The  set of $X$-colorings of $K$ is denoted by $\Col_X(K)$.

    Biquandle colorings give rise to several invariants for knotted structures. The \textit{biquandle counting invariant}, for instance, is defined as the number of colorings of a diagram with a choosen biquandle. For (virtual) knotoids, the existence of distinguished tail and head semiarcs provide a partition of biquandle colorings by recording the colors of these distinguished semiarcs. An \textit{$X_{ij}$-coloring} of a virtual knotoid diagram $K$ is an $X$-coloring of $K$ where the tail and the head semiarcs of $K$ are colored by $x_i$ and $x_j$ in $X$, respectively. The set of such colorings is denoted by $\Col_X^{ij}(K)$. The number of $X_{ij}$-colorings of $K$ can be stored in a matrix, called the \textit{biquandle counting matrix} invariant, denoted by $\mathcal{M}_X(K)$. For classical knotoids \cite{gugumcu2021biquandle} and virtual knotoids \cite{gugumcu2026biquandle}, it was shown that the biquandle counting matrices satisfy
    \begin{align}
    \bcm{K_1\#K_2}=\bcm{K_1}\bcm{K_2}.
    \label{eq:bcmcs}
    \end{align}
    
    Biquandle coloring invariants are further enhanced by assigning algebraic data to biquandle-colored
    diagrams. \textit{Biquandle brackets}
    \cite{nelson2017quantum} and their extension, \textit{biquandle virtual
    brackets} \cite{nelson2019biquandle}, provide algebraic structures together with state-sum formulas which assign values to biquandle-colored diagrams. Special cases of these structures and resulting invariants include skein invariants such as the Kauffman bracket \cite{kauffman1987state}. In \cite{gugumcu2026biquandle}, biquandle virtual brackets were utilized to obtain new invariants of virtual knotoids. Moreoever, as in the case of biquandle colorings, biquandle virtual bracket values were organized with respect to the endpoint semiarc colors and \textit{biquandle virtual bracket matrix} invariant of virtual knotoids was instroduced.

    The \textit{product} of knotoids, introduced by Turaev in~\cite{turaev2012knotoids}, is a binary operation obtained by joining the head of the first knotoid diagram to the
    tail of the second. Throughout the paper, we refer to this product as the \textit{connected sum} of knotoids and denote it by $\,\#\,$. The construction of the connected sum extends to virtual knotoids \cite{gugumcu2026biquandle} and it was shown that the set of virtual knotoids endowed with connected sum operation forms a monoid. The aim of this paper is to investigate the behavior of biquandle-based invariants of virtual knotoids, such as the fundamental biquandle, biquandle coloring invariants and biquandle virtual bracket invariants, under connected sum.
    
    Cattabriga and Horvat developed a categorical approach to tangles and their geometric compositions in \cite{cattabriga2020} by introducing the \textit{bordered-quandle category}. In this approach, the boundary data of tangles is represented by free quandles and geometric composition of tangles corresponds to composition of cospans by pushout in the bordered-quandle category. 
    
    It is well known that every monoid can be considered as a category with one object. Hence, the monoid of virtual knotoids under connected sum operation can be considered as a category with one object where the morphisms are virtual knotoids, and the composition is their connected sum operation. Motivated by the work in \cite{cattabriga2020}, we show that the connected sum in this category is reflected algebraically by a pushout in the biquandle category. More precisely, if $F_1$ denotes the free biquandle on one generator, then
    \[
    \mathcal B(K_1\#K_2)\cong\mathcal B(K_1)*_{F_1}\mathcal B(K_2),
    \]
    where the two structure maps in amalgamation on the right-hand side send the generator of $F_1$ to the head element of $K_1$ and the tail element of $K_2$, respectively.

    We show that the categorical description also clarifies the behavior of biquandle colorings of virtual knotoids under connected sum. Applying the contravariant $\Hom$ functor with target biquandle $X$ turns
    the aforementioned pushout into a pullback:
    \[
    \Col_X(K_1\#K_2)
    \cong
    \Col_X(K_1)\times_X\Col_X(K_2),
    \]
    in the category of sets where pullback compatibility requires the head semiarc color of $K_1$ to be equal the tail semiarc color of $K_2$. For fixed elements $x_i,x_k\in X$, this gives
    \[
    \Col_X^{ik}(K_1\#K_2)
    \cong
    \bigsqcup_{x_j\in X}
    \Col_X^{ij}(K_1)\times
    \Col_X^{jk}(K_2).
    \]
    Taking cardinalities recovers the previously known biquandle counting matrix
    formula~\eqref{eq:bcmcs}. Thus, the multiplicativity of biquandle counting matrix follows as a natural consequence of the pushout description of the fundamental biquandle. Hence, investigating the behavior of the fundamental biquandle under connected sum leads to an explanation of this matrix product formula.

    In \cite{pflume}, Pflume introduced \textit{pointed biquandles} and the \textit{fundamental pointed biquandle} of a linkoid. In this structure, generators corresponding to endpoint semiarcs are retained as distinguished elements in the fundamental biquandle represetation. We extend the notion of pointed fundamental biquandles to virtual knotoids and use this pointed approach to specify generators which are getting identified in connected sum.

    The second aim of the paper is to determine how the biquandle virtual bracket values and induced invariants behave under connected
    sum. A \textit{state} of a virtual knotoid diagram $K$ is obtained by \textit{smoothing} each classical crossing of $K$. Every state of $K$ contains a unique open component, together with a collection of closed components. We introduce the \emph{normalized} biquandle virtual bracket state-sum obtained by replacing the usual component factor $\;\delta^{|s|}\;$ in the construction in \cite{nelson2019biquandle} by $\;\delta^{|s|-1},\;$
    so that the exponent records only the number of closed components in $s$. Both conventions appear in the literature on the Kauffman bracket, where they yield invariants with the same distinguishing power. More generally, the normalized and unnormalized biquandle virtual bracket invariants have the same distinguishing power whenever $\delta$ is not a zero divisor. However, when $\delta$ is a zero divisor, the normalized invariant may be strictly stronger. We demonstrate this by an explicit example. 
    
    We then establish that the normalized biquandle virtual bracket value is multiplicative under connected sum:
    \[
    \bvbn{(K_1\#K_2)f}=\bvbn{(K_1)_{f_1}}\bvbn{(K_2)_{f_2}},
    \]
    where $f$ is the $X$-coloring of $K_1\#K_2$ corresponding to compatible $X$-colorings $f_1$ and $f_2$ of $K_1$ and $K_2$, respectively.
    We combine this multiplicativity with the pullback description of biquandle colorings of the connected sum and obtain connected-sum formulas for the normalized biquandle virtual
    bracket multiset and its partitioned refinement with respect to endpoint semiarc colors over an arbitrary
    coefficient ring.

    When the coefficient ring of a biquandle virtual bracket is a number ring, these multiset invariants can be encoded by polynomials \cite{nelson2019biquandle, gugumcu2026biquandle}. Since normalized biquandle virtual bracket values are multiplicative under connected sum, the bracket values $a$ and $b$ contribute the monomial $u^{ab}$, rather than $u^{a+b}$. We therefore introduce the product $\star$, defined on monomials by $u^a\star u^b=u^{ab}$, and extend it distributively. Applying this operation entrywise in the usual matrix product fashion defines a product $\odot$ on normalized biquandle virtual bracket matrices. Then we prove that,
    \[
    \bvbmn{K_1\#K_2}=\bvbmn{K_1}\odot\bvbmn{K_2}.
    \]

    The paper is organized as follows. In Section~\ref{sec:prelims}, we review the ingredients used throughout the paper. We first review virtual knotoids and their connected-sum operation in Section~\ref{sec:vknotoid-prelims}. We then recall biquandles and biquandle-based invariants of virtual knotoids in Section~\ref{sec:biquandle-prelims}. Finally, in Section~\ref{sec:boundary-pointed-pushouts}, we review the necessary pointed and categorical background of biquandles. We present our results in Section~\ref{sec:connected-sum-results}. We begin with a description of the connected-sum operation using virtual connection arcs in Section~\ref{sec:new-connected-sum}. Then, in Section~\ref{sec:biquandle-behavior}, we determine the behavior of the fundamental biquandle under connected sum, give its pushout interpretation in the biquandle category and study the results of this categorical viewpoint. Finally, in Section~\ref{sec:bracket-behavior}, we introduce the normalized biquandle virtual bracket state-sum and establish connected-sum formulas for the induced multiset, polynomial, and matrix invariants. Section~\ref{sec:discussion} concludes with further remarks.

\section{Preliminaries}
\label{sec:prelims}

\subsection{Virtual knotoids and connected sum}
In this section, we recall the basics of virtual knotoid theory and the connected sum operation of virtual knotoids.
\label{sec:vknotoid-prelims}

\subsubsection{Virtual knotoids}
    Virtual knotoid theory, introduced by Gügümcü and Kauffman in \cite{gugumcu2017new}, combines the open-ended structure of knotoids \cite{turaev2012knotoids} with the virtual framework of virtual knot theory \cite{kauffman1999virtual}.

    \begin{figure}[h!]
        \centering
        \includegraphics[width=0.08\textwidth]{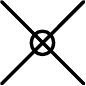}
        \caption{A virtual crossing.}
        \label{fig:virtual-crossing}
    \end{figure}
    
    \begin{definition}
        A \textit{virtual crossing} is a transverse double point marked by a small circle
        and carrying no over/under crossing information (see
        Figure~\ref{fig:virtual-crossing}). A \textit{virtual knotoid diagram} $K$ is a generic
        immersion of the oriented interval $I=[0,1]$ into $S^2$ whose transverse double
        points are either classical crossings, equipped with over/under
        information, or virtual crossings (see Figure~\ref{vdiagrams}). The images of
        $0$ and $1$ are called the \textit{tail} and the \textit{head} of $K$, respectively. A \textit{trivial virtual knotoid diagram} is a virtual knotoid diagram which contains no classical crossings. 
    \end{definition}

    \begin{figure}[h!]
        \centering
        \begin{subfigure}[b]{0.21\textwidth}
            \centering
            \includegraphics[width=\linewidth]{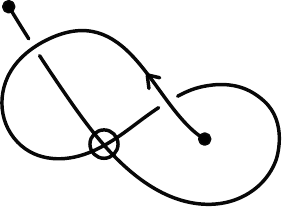}
            \caption{$K_1$}
        \end{subfigure}
        \hspace{1.8cm}
        \begin{subfigure}[b]{0.18\textwidth}
            \centering
            \includegraphics[width=\linewidth]{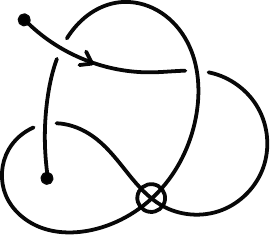}
            \caption{$K_2$}
        \end{subfigure}
        \caption{Two virtual knotoid diagrams.}
        \label{vdiagrams}
    \end{figure}

    \begin{definition}
        A \textit{semiarc} of a virtual knotoid diagram $K$ is a segment of the diagram between two consecutive crossings, between an endpoint and an adjacent crossing, or the whole diagram when the diagram contains no crossings. The semiarcs containing the tail and the head are called the \textit{tail semiarc} and \textit{head semiarc} of $K$, respectively.
    \end{definition}

    Virtual knotoid diagrams are considered up to
    isotopies of $S^2$ and \textit{generalized Reidemeister moves} shown in Figure~\ref{fig:reidemeister}. The resulting equivalence classes are called \textit{virtual knotoids}. The \textit{detour move} allows a segment in a virtual knotoid diagram containing only
    virtual crossings to be replaced by another segment with the same
    endpoints, where all intersections of the new segment with the diagram are virtual crossings. The moves involving virtual crossings in Figure~\ref{fig:reidemeister} are generated by the detour move.

    \begin{figure}[h!]
        \centering
        \includegraphics[width=0.43\linewidth]{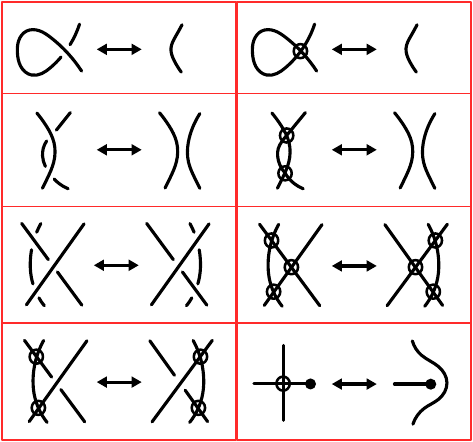}
        \caption{Generalized Reidemeister moves for virtual knotoid diagrams.}
        \label{fig:reidemeister}
    \end{figure}

    The move shown in the bottom of the right-hand side of Figure~\ref{fig:reidemeister} is called the \textit{virtual endpoint move}. If the crossing adjacent to an endpoint is virtual, this move allows the endpoint is to move past the strand. Conversely, an endpoint can pass through a strand creating a virtual crossing by the virtual endpoint move. However, it is not allowed to move an endpoint under or over a strand through a classical crossing. This move is called the \textit{forbidden move} and shown in Figure~\ref{fig:forbidden}.

    \begin{figure}[h!]
        \centering
        \includegraphics[width=0.43\linewidth]{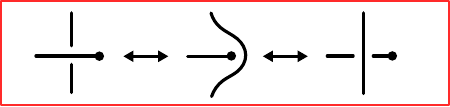}
        \caption{Forbidden move for virtual knotoid diagrams.}
        \label{fig:forbidden}
    \end{figure}

    A knotoid diagram in $S^2$ can be considered as a virtual knotoid diagram
    with no virtual crossings. It was shown in \cite{gugumcu2025virtual} that the map from
    knotoids in $S^2$ to virtual knotoids induced by inclusion map of diagrams is injective, and hence knotoid theory in $S^2$ embeds into
    virtual knotoid theory. 
    
    Virtual knotoids also admit geometric interpretations as knotoids in surfaces \cite{gugumcu2017new} and as rail arcs in thickened surfaces \cite{gugumcu2025virtual} considered up to stabilization and destabilization.
    
\subsubsection{Connected sum of virtual knotoids}

    By the connected sum operation of knotoids, we mean the product introduced by Turaev in \cite{turaev2012knotoids}. This operation was extended to virtual knotoids \cite{gugumcu2026biquandle}. We recall the definition for virtual knotoids, which also applies to knotoids in $S^2$.

    \begin{definition}
		Let $K_1$ and $K_2$ be two virtual knotoid diagrams with tails $t_1,t_2$ and
		heads $h_1,h_2$. Let $U$ and $V$ be closed disk neighborhoods of
		$h_1$ and $t_2$, respectively, such that each disk meets the corresponding
		virtual knotoid diagram in a single radial arc. Remove the interiors of
		$U$ and $V$ and identify the resulting boundary circles by an orientation-reversing homeomorphism $ \varphi:\partial U\to\partial V$
		which sends the point $K_1\cap \partial U$ to the point $K_2\cap \partial V$. The resulting virtual knotoid diagram is the connected sum of $K_1$ and $K_2$, which we denote by $K_1\# K_2$. See Figure~\ref{fig:connected-sum} for an illustration.
	\end{definition}

    \begin{figure}[h!]
		\centering
		\includegraphics[width=0.35\textwidth]{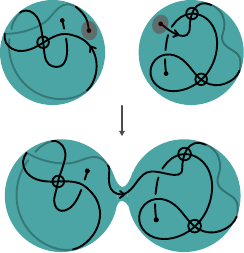}
		\caption{The connected sum $K_1\#K_2$ of virtual knotoid diagrams $K_1$ and $K_2$.}
		\label{fig:connected-sum}
	\end{figure}

    The connected sum construction is well defined up to the equivalence of knotoids \cite{turaev2012knotoids} and virtual knotoids \cite{gugumcu2026biquandle}. Thus, if $\mathcal{K}^{\mathrm{v}}$ denotes the set of virtual knotoids, connected sum induces a well-defined operation
    \begin{align*}\#:\mathcal{K}^{\mathrm{v}}\times\mathcal{K}^{\mathrm{v}}
    &\longrightarrow
    \mathcal{K}^{\mathrm{v}}\\
    ([K_1],[K_2])
    &\longmapsto
    [K_1\#K_2].
    \end{align*}
    The connected sum operation is associative, has the trivial virtual knotoid as its identity element, and hence the set of virtual knotoids forms a monoid under connected sum. However, unlike the connected sum of knots, the operation is not commutative in (virtual) knotoid setting.

\subsection{Biquandle-based invariants of virtual knotoids}
\label{sec:biquandle-prelims}
    
    In this section, we review biquandles and biquandle coloring invariants of virtual knotoids. Then, we continue with their enhancements obtained by utilizing biquandle virtual brackets.

\subsubsection{Biquandles and biquandle colorings}

    We begin with the category of biquandles.
    
    \begin{definition}
        A \textit{biquandle} is a set $X$ with two binary operations $\,\un,\ovr : X\times X\rightarrow X$
        such that, for all $x,y,z\in X$:
        \begin{enumerate}
        \item [i.] $x\un x=x\ovr x,$
        \item [ii.] the maps
        $\;\alpha_y(x)=x\un y,\;
        \gamma_y(x)=x\ovr y,\; \text{and }
        S(x,y)=\bigl(y\ovr x,\ x\un y\bigr)\;$
        are bijections,
        \item [iii.] the exchange laws hold:
        \begin{align*}
        (x\un y)\un(z\un y)=
        (x\un z)\un(y\ovr z),\\
        (x\un y)\ovr(z\un y)=
        (x\ovr z)\un(y\ovr z),\\
        (x\ovr y)\ovr(z\ovr y)=
        (x\ovr z)\ovr(y\un z).
        \end{align*}
        \end{enumerate}
        \end{definition}
    
        \begin{definition}
        Let $X$ and $Y$ be biquandles. A map $f:X\rightarrow Y$ is a \textit{biquandle homomorphism} if for all $x,y\in X$
        \[
        f(x\un y)=f(x)\un f(y),\qquad
        f(x\ovr y)=f(x)\ovr f(y).
        \]
    \end{definition}

    Biquandles and biquandle homomorphisms form a category, which we denote by
    $\mathbf{Bq}$. The set of biquandle homomorphisms from $X$ to $Y$ is denoted by $\Hom_{\mathbf{Bq}}(X,Y).$

    \begin{definition}
    Let $K$ be a virtual knotoid diagram with labeled semiarcs. Let $L(K)$ denote the set of these labels and $R(K)$ denote the set of relations among the labels of the semiarcs at the crossings of $K$, as shown in Figure \ref{fig:biquandle-coloring}. The
    \textit{fundamental biquandle} of $K$, denoted by $\mathcal{B}(K)$, is the biquandle generated by the elements of $L(K)$ considered up to the relations $R(K)$. That is,
    \[\mathcal{B}(K)=\langle L(K)\;|\; R(K)\rangle.\]
    \end{definition}

    \begin{figure}[h!]
        \centering
        \includegraphics[width=0.5\textwidth]{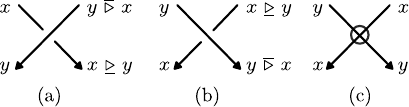}
        \caption{Biquandle labelings of semiarcs around classical and virtual crossings.}
        \label{fig:biquandle-coloring}
    \end{figure}

    It was shown in \cite{gugumcu2026biquandle} that $\mathcal{B}(K)$ is an invariant  of virtual knotoids up to biquandle isomorphisms.

    \begin{definition}
        Let $K$ be a virtual knotoid diagram and let $X$ be a biquandle. An \textit{$X$-coloring} of $K$ is a biquandle homomorphism $\;f:\mathcal B(K)\rightarrow X.$ Hence, the set of $X$-colorings of $K$ is 
        \[
        \Col_X(K)=\Hom_{\mathbf{Bq}}\left(\mathcal B(K),X\right).
        \]
        A virtual knotoid diagram $K$ equipped with an $X$-coloring $f$ is denoted by $K_f$.
    \end{definition}

    \begin{remark}
        The convention for the labeling rules in Figure~\ref{fig:biquandle-coloring} is from left to right as in \cite{cho2019quandle}. That is, the labels on the left-hand side of a crossing are considered as the \textit{input} labels, and those on the right-hand side as the \textit{output} labels.
    \end{remark}

    Let $D$ be a local region of a virtual knotoid diagram $K$ where a generalized Reidemeister move is performed. The biquandle axioms guarantee that a semiarc labeling in $D$ before the move correspond to a unique semiarc labeling in $D$ after the move. Hence, the number of biquandle colorings is preserved under generalized Reidemeister moves, providing an invariant of virtual knotoids \cite{gugumcu2026biquandle}. This invariant is called the \textit{biquandle counting invariant} of virtual knotoids and is denoted by $\Phi_X^{\Z}(K)$. Thus we have,
    \[\Phi_X^{\Z}(K)=|\Col_X(K)|=|\Hom_{\mathbf{Bq}}(\mathcal{B}(K),X)|.\]

    The set of $X$-colorings of a virtual knotoid diagram $K$ can be partitioned according to the colors assigned to the tail and head semiarcs of $K$.

    \begin{definition}
        Let $X=\{x_1,\ldots,x_n\}$ be a biquandle and $K$ a virtual knotoid diagram. Let $a_t$ and $a_h$ be the generators of the fundamental biquandle $\mathcal{B}(K)$ corresponding to the tail and head semiarcs of $K$, respectively. For $i,j\in\{1,\ldots,n\}$, an \textit{$X_{ij}$-coloring} of $K$ is an $X$-coloring $\;f\in \Hom_{\mathbf{Bq}}\bigl(\mathcal B(K),X\bigr)\;$ of $K$ such that $\;f(a_t)=x_i,$ and $f(a_h)=x_j.\;$ The set of $X_{ij}$-colorings of $K$ is denoted by $\Col_X^{ij}(K)$. Thus,
        \[ \Col_X^{ij}(K)=\Hom_{\mathbf{Bq}}^{ij}\bigl(\mathcal B(K),X\bigr)=\{f\in \Hom_{\mathbf{Bq}}(\mathcal{B}(K),X)\;|\; f(a_t)=x_i, \; f(a_h)=x_j\}.
        \]
    \end{definition}

    It is clear that $\Col_X(K)=\bigsqcup_{i,j=1}^{n}\Col_X^{ij}(K),$
    and so $\Phi_X^{\mathbb Z}(K)=\sum_{i,j=1}^{n}\left|\Col_X^{ij}(K)\right|.$ The numbers of $X_{ij}$-colorings can be stored in a matrix.

    \begin{definition}
        Let $X=\{x_1,\ldots,x_n\}$ be a biquandle and $K$ a virtual knotoid diagram. The \textit{biquandle counting matrix} of $K$ is the $n\times n$ matrix $\bcm{K}=\left(m_{ij}(K)\right),$ where $m_{ij}(K)=\left|\Col_X^{ij}(K)\right|.$ That is, the $(i,j)$-entry of $\bcm{K}$ is the number of $X_{ij}$-colorings of $K$.
    \end{definition}

    It was shown in \cite{gugumcu2026biquandle} that the biquandle counting matrix is an invariant of virtual knotoids. Moreover, it is clear that
    \[
    \Phi_X^{\Z}(K)
    =
    \sum_{i,j=1}^{n}
    \bcm{K}_{i,j},
    \]
    meaning that the biquandle counting matrix is an enhancement of the biquandle counting invariant.
    
\subsubsection{Biquandle virtual brackets and virtual knotoids}
    The main idea in the construction of invariants we review in this section is to assign bracket values to the biquandle colored virtual knotoid diagrams. The assignment of such values are through biquandle virtual bracket state-sum formulas. There are twenty seven defining axioms for a biquandle virtual bracket which, all together, guarantee that the assigned values to biquandle colored diagrams are preserved under generalized Reidemeister moves. We omit these axioms here and refer the reader to \cite{nelson2019biquandle} for the full list of axioms and a detailed derivation of these axioms.
    
    \begin{definition}
        Let $X$ be a biquandle and $R$ a commutative ring with identity. A
        \emph{biquandle virtual bracket} $\beta$ over $R$ with respect to $X$
        consists of coefficient maps
        \[
        A,B,V,C,D,U:X\times X\longrightarrow R,
        \]
        and elements $\delta\in R$ and $\omega\in R^\times,$ satisfying the biquandle virtual bracket axioms given in \cite{nelson2019biquandle}. For notational convenience, we write $ F_{x,y}=F(x,y),$ where $F\in\{A,B,V,C,D,U\},$ 
    \end{definition}

    Let $K$ be a virtual knotoid diagram and $f\in\Col_X(K)$ an $X$-coloring of $K$. At each classical crossing of $K$, there are three possible local replacements: a \textit{vertical smoothing}, a \textit{horizontal smoothing}, or a \textit{virtual smoothing}, as shown in Figure~\ref{fig:smoothings}. A
    \textit{state} of $K$ is obtained by smoothing each classical crossing in $K$. Each state of a virtual knotoid diagram consists of an open segment component, and a finite number of closed components possibly containing virtual crossings.

    \begin{figure}[h!]
        \centering
        \begin{overpic}[width=0.33\linewidth]{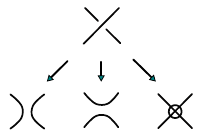}
        \end{overpic}
        \caption{Vertical, horizontal and virtual smoothings of a classical crossing from left to right.}
        \label{fig:smoothings}
    \end{figure}

    At a positive downward-oriented classical crossing with input colors $x,y$, we assign
    the coefficients
    \[
    \text{vertical smoothing}\mapsto A_{x,y},
    \qquad
    \text{horizontal smoothing}\mapsto B_{x,y},
    \qquad
    \text{virtual smoothing}\mapsto V_{x,y}.
    \]
    At a negative downward-oriented classical crossing,
    the corresponding coefficients are
    \[
    \text{vertical smoothing}\mapsto C_{x,y},
    \qquad
    \text{horizontal smoothing}\mapsto D_{x,y},
    \qquad
    \text{virtual smoothing}\mapsto U_{x,y}.
    \]
    With these assignments, the skein relations of biquandle virtual brackets are as shown in Figure~\ref{fig:x-virtual-bracket-smoothings}.

    \begin{figure}[h!]
        \centering
        \includegraphics[width=0.65\textwidth]{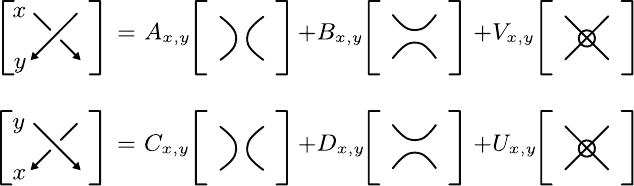}
        \vspace{0.2cm}
        \caption{Skein relations of biquandle virtual brackets.}
        \label{fig:x-virtual-bracket-smoothings}
    \end{figure}

    \begin{example}
    \label{ex: state}
    A biquandle colored virtual knotoid diagram $K$ is shown on the left-hand side of
    Figure~\ref{fig:smoothed}. On the
    right-hand side, we show one state $s$ of $K$ and the corresponding
    smoothing coefficients.
    
    \begin{figure}[h!]
        \centering
        \begin{overpic}[width=0.5\linewidth]{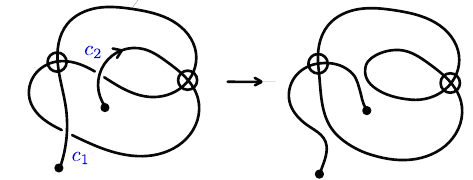}
            \put(17,15){\small $x$}
            \put(30.5,1){\small $y$}
            \put(2,15){\small $z$}
            \put(33,13){\small $s$}
            \put(9,5){\small $t$}
            \put(58,6){\scalebox{0.9}{$D_{z,s}$}}
            \put(80,20){\scalebox{0.9}{$D_{s,x}$}}
        \end{overpic}
        \caption{A biquandle colored virtual knotoid diagram and one of its states with the
        corresponding smoothing coefficients.}
        \label{fig:smoothed}
    \end{figure}
    \end{example}

    Let $S(K)$ denote the set of states of a virtual knotoid diagram $K$. For $s\in S(K),$ let $|s|$ denote the number of connected components of $s$, including
    its unique open segment component. Let $\operatorname{w}(K)$ denote the writhe of $K$. Let $f$ be an $X$-coloring of $K$. For a classical crossing $v$ in $K$, let $c_{s,f}(v)$ denote the local coefficient
    determined by the sign, input colors $x,y$ and the smoothing type of $v$ in $s$. Hence, $\;
    c_{s,f}(v)\in \{A_{x,y},B_{x,y},V_{x,y},C_{x,y},D_{x,y},U_{x,y}\}$.

    \begin{definition}
        Let $K$ be a virtual knotoid diagram, $X$ a biquandle, and
        $\beta$ a biquandle virtual bracket. For an $X$-coloring $f\in\Col_X(K),$
        the \textit{biquandle virtual bracket value} of $K_f$ is the value obtained from the state-sum
        \[
        \bvb{K_f}=\omega^{-\operatorname{w}(K)} \sum_{s\in S(K)}\left(
        \prod_{v\in C(K)}c_{s,f}(v) \right)\delta^{|s|},
        \]
        where $C(K)$ denotes the set of classical crossings of $K$.
    \end{definition}

    \begin{example}
        Consider the virtual knotoid diagram $K$ and its one state $s$ given in Example~\ref{ex: state}. The choice of smoothings at the classical crossings of $K$ produces the coefficients $D_{z,s}$ and $D_{s,x}$. The resulting state $s$ consists of two components. Hence, the contribution of $s$ to the bracket value is $D_{z,s}D_{s,x}\delta^2.$ Summing these contributions for all possible states of $K$ produces the biquandle virtual bracket value $\bvb{K_f}$.
    \end{example}

    \begin{example}
    \label{ex:kaufmann}
        Consider the biquandle $X=\{1\}$ and
        the biquandle virtual bracket $\beta$ over
        $R=\mathbb{Z}[A,A^{-1}]$ given by
        \[
            A_{1,1}=D_{1,1}=A,\quad
            B_{1,1}=C_{1,1}=A^{-1},\quad
            V_{1,1}=U_{1,1}=0,
        \]
        with $\delta=-A^2-A^{-2}$ and $\omega=-A^3$.
        Every virtual knotoid diagram $K$ admits a unique trivial $X$-coloring $f$ where each semiarc is colored by $1$. Since $V_{1,1}=U_{1,1}=0$,
        only the two classical smoothing choices contribute to the state-sum. Thus,
        \[
            \bvb{K_f}
            =
            (-A^3)^{-\operatorname{w}(K)}
            \sum_{s\in S(K)}
            A^{a(s)-b(s)}
            (-A^2-A^{-2})^{|s|},
        \]
        where $a(s)$ and $b(s)$ denote the numbers of
        smoothings with coefficients $A$ and $A^{-1}$,
        respectively. This recovers the
        Kauffman bracket with writhe correction where a factor of $\delta$ assigned
        to each state component.
    \end{example}

    Let $K$ be a virtual knotoid diagram and $\;1_{\mathcal{B}(K)}:
    \mathcal{B}(K)\longrightarrow\mathcal{B}(K)\;$ 
    denote the identity biquandle homomorphism of its fundamental biquandle. This homomorphism determines a
    $\mathcal B(K)$-coloring of $K$ where each semiarc of $K$ is colored by its corresponding generator in $\mathcal{B}(K)$. 

    \begin{definition}
        Let $K$ be a virtual knotoid diagram. The
        \textit{fundamental biquandle virtual bracket expression} of $K$, denoted by $\fbvb{K}$, is the formal biquandle virtual bracket expression obtained from the identity coloring $1_{\mathcal B(K)}$ of $K$.
    \end{definition}
    
    This formal expression is quite useful in computations. For a given biquandle $X$ and a virtual knotoid diagram $K$, suppose that the semiarcs of $K$ correspond to generators $a_1,\ldots,a_m$ of $\mathcal B(K)$, and let $f\in\Col_X(K).$ For a chosen biquandle virtual bracket $\beta$ with respect to $X$, the corresponding bracket value
    $\bvb{K_f}$ can be obtained from $\fbvb{K}$ by replacing every formal coefficient $F_{a_r,a_s}$ with $F_{f(a_r),f(a_s)},$ where $F\in\{A,B,V,C,D,U\},$ 
    and then evaluating the coefficient maps, $\delta$ and $\omega$ according to $\beta$. Thus, $\fbvb{K}$ provides a formal bracket expression for $K$ from which the bracket value associated to any $X$-coloring of $K$ for any choosen biquandle virtual bracket $\beta$ can be obtained by evaluating this formal expression accordingly.

    \begin{definition}
        Let $K$ be a virtual knotoid diagram, $X$ a finite biquandle, and
        $\beta$ a biquandle virtual bracket. The \textit{biquandle virtual bracket multiset} of $K$ is
        \[
        \bvbmult{K}=\left\{\left\{
        \bvb{K_f}
        \;\middle|\;
        f\in\Col_X(K)
        \right\}\right\}.
        \]
        That is, the biquandle virtual bracket multiset consists of biquandle virtual bracket values of all possible $X$-colorings of $K$.
    \end{definition}

    It was shown in \cite{gugumcu2026biquandle} that the biquandle virtual bracket multiset is an invariant of virtual knotoids. Its multiset cardinality is equal to the biquandle counting
    invariant, $\left|\bvbmult{K}\right|=\Phi_X^{\mathbb Z}(K),$
    so it is an enhancement of the biquandle counting invariant.
    When $R$ is a number ring, this multiset can be encoded by the polynomial 
    \[
    \Phi_X^\beta(K)=\sum_{f\in\Col_X(K)} u^{\bvb{K_f}},
    \]
    where $u$ is a formal variable. This polynomial is called the
    \textit{biquandle virtual bracket polynomial}.

    The partition of $X$-colorings of a virtual knotoid diagram $K$ as $X_{ij}$-colorings provides an enhancement of the biquandle virtual bracket multiset. For $x_i,x_j\in X$, define
    \[
    \Phi_{X_{ij}}^{\beta,M}(K)=\left\{\left\{\bvb{K_f}\;|\;f\in\Col_X^{ij}(K)\right\}\right\}.
    \]
    Since $\Col_X(K)=\bigsqcup_{i,j=1}^{n}
    \Col_X^{ij}(K),$
    we have
    \[
    \bvbmult{K}=\biguplus_{i,j=1}^{n}\Phi_{X_{ij}}^{\beta,M}(K),
    \]
    where $\biguplus$ denotes the multiset union where multiplicities of equal elements are summed. Thus, the family
    \[
    \left\{
    \Phi_{X_{ij}}^{\beta,M}(K)
    \right\}_{i,j=1}^{n}
    \]
    records the biquandle virtual bracket values for each $\Col_X^{ij}(K)\subseteq \Col_X(K)$ separately. This family of multisets is an invariant of virtual knotoids \cite{gugumcu2026biquandle}.
    
    When $R$ is a number ring, this family of multisets can be represented by matrix with polynomial entries.

    \begin{definition}
        Let $X=\{x_1,\ldots,x_n\}$ be a finite biquandle, $K$ a virtual
        knotoid diagram, and $\beta$ a biquandle virtual bracket over a
        commutative unitary number ring $R$. The
        \textit{biquandle virtual bracket matrix} of $K$ with respect to $X$ and
        $\beta$ is the $n\times n$ matrix
        \[
        \bvbm{K}=\left(\bvbm{K}_{ij}\right)_{1\leq i,j\leq n},
        \quad \text{where }\;
        \bvbm{K}_{ij}=\sum_{f\in\Col_X^{ij}(K)} u^{\bvb{K_f}}.
        \]
        Thus, the $(i,j)$-entry of this matrix records the biquandle virtual bracket polynomial of $X_{ij}$-colorings of $K$.
    \end{definition}

    It was shown in \cite{gugumcu2026biquandle} that the biquandle virtual bracket matrix is an enhancement of the biquandle counting matrix and the biquandle virtual bracket polynomial of virtual knotoids, and is strictly stronger than both.

\subsection{Pointed biquandles and categorical preliminaries}
\label{sec:boundary-pointed-pushouts}

    The presence of distinguished endpoint semiarcs in (virtual) knotoids and keeping track of information assigned to these semiarcs is quite useful in constructing strong invariants as we have seen in the previous section. A convenient way of recording endpoint information of linkoids was introduced utilizing \textit{pointed quandles} \cite{gugumcu2026pointed} and \textit{pointed biquandles}
    \cite{pflume}. We recall the notion of pointed biquandles.

    \begin{definition}
        An \textit{$n$-pointed biquandle} is an ordered tuple $(X,x_1,\ldots,x_n),$
        where $X$ is a biquandle and $x_1,\ldots,x_n$ are distinguished
        elements of $X$. An \textit{$n$-pointed biquandle homomorphism}
        \[
        f:(X,x_1,\ldots,x_n)\longrightarrow
        (X',x_1',\ldots,x_n')
        \]
        is a biquandle homomorphism $f:X\to X'$ such that $f(x_i)=x_i'$
        for all $i=1,\ldots,n$.
    \end{definition}

    The $n$-pointed biquandles and their homomorphisms form a category, denoted by
    $\mathbf{PBq}_n$.
    
    Let $L$ be a linkoid diagram that contains $n$ open components. The generators corresponding to endpoint semiarcs of $L$ in its fundamental biquandle determine distinguished
    elements. This
    leads to the construction of the $(2n)$-pointed biquandle, \[\left(
    \mathcal{B}(L),
    t_1,h_1,\ldots,t_n,h_n
    \right),\]
    where $\mathcal{B}(L)$ is the fundamental biquandle of $L$ and $t_i$ and $h_i$ are generators corresponding to the tail and head semiarcs of the
    $i$-th open component of $L$. This $(2n)$-pointed biquandle associated to $L$ is called the \textit{fundamental pointed biquandle} of $L$, and denoted by $\mathcal{PB}(L)$. $\mathcal{PB}(L)$ is an invariant of linkoids.

    We continue with recalling the notions of pushout and pullback in a category.

    \begin{definition}
        Let $\mathcal{C}$ be a category. A \textit{pushout} of a \textit{span} $\;A\xleftarrow{\,f\,}C\xrightarrow{\,g\,}B\;$ in $\mathcal{C}$ consists
        of an object $P\in \mathcal{C}$ together with morphisms
        $
        \lambda_A:A\longrightarrow P,
        \text{ and }
        \lambda_B:B\longrightarrow P
        $
        satisfying,
        \begin{itemize}
            \item $\lambda_A\circ f=\lambda_B\circ g,$
            \item the universal property: for every object $Y\in \mathcal{C}$ and morphisms $\;\varphi_A:A\longrightarrow Y, \text{ and } \varphi_B:B\longrightarrow Y$ with $ \varphi_A\circ f=\varphi_B\circ g,\;$
            there exists a unique morphism $\;\varphi:P\longrightarrow Y\;$ such that $\;\varphi\circ\lambda_A=\varphi_A, \text{ and }\varphi\circ\lambda_B=\varphi_B.$
        \end{itemize}
        Equivalently, the following diagram commutes,
        \[
        \begin{tikzcd}
        C \arrow[r,"g"] \arrow[d,"f"'] & B \arrow[d,"\lambda_B"] \\
        A \arrow[r,"\lambda_A"'] & P
        \end{tikzcd}
        \]
        and it is universal among all such commutative squares containing the span $\;A\xleftarrow{\,f\,}C\xrightarrow{\,g\,}B$.
        This pushout $P$ is denoted by $A*_{(C,f,g)}B$. 
    
        Dually, a \textit{pullback} of a \textit{cospan} $A\xrightarrow{f}C\xleftarrow{g}B$
        is an object $P$ with morphisms
        $\;P\longrightarrow A,
        \text{ and }
        P\longrightarrow B
        \;$
        satisfying the dual of the universal property of pushout. In $\mathbf{Set}$ category, the pullback is
        the \emph{fiber product}
        \[
        A\times_C B
        =
        \{(a,b)\in A\times B\mid f(a)=g(b)\},
        \]
        and $f(a)=g(b)$ is called the \textit{compatibility condition}.
    \end{definition}
    It follows from the definitions that a pushout square in a category $\mathcal{C}$ becomes a pullback square in its opposite category $\mathcal{C}^{\operatorname{op}}$.

    In~\cite{cattabriga2020}, Cattabriga and Horvat introduced  a categorical
    framework for fundamental quandles of tangles. They define a bordered quandle
    category, where the objects are finitely generated free quandles. A morphism between two objects $F_{1}$ and $F_2$ in this category is a \textit{cospan} $\;F_1\xrightarrow{i_1} Q\xleftarrow{i_2} F_2\;$ where $Q$ is a quandle and $i_1$, $i_2$ are quandle homomorphisms. If $F_1\xrightarrow{i_1} Q_1\xleftarrow{j} F
    \text{ and }
    F\xrightarrow{k} Q_2\xleftarrow{i_2} F_2$
    are two cospans, their composition is represented by
    \[
    F_{1}
    \xrightarrow{i_1}
    Q_1*_{(F,j,k)}Q_2
    \xleftarrow{i_2}
    F_{2},
    \]
    where $Q_1*_{(F,j,k)}Q_2$ is the pushout over the common boundary quandle $F$ of these two cospans.
    Under this construction, the fundamental quandle gives a functor from the
    oriented tangle category to the bordered quandle category. Thus,
    composition of tangles is realized by the pushout of their
    fundamental quandles in this bordered quandle category. 

    In our categorical constructions, we use a free biquandle on one generator.

    \begin{definition}
        A \textit{free biquandle on one generator} consists of a
        biquandle $F_1$ and an element $a\in F_1$ such that, for every biquandle $X$ and every element $x\in X$, there exists a unique biquandle homomorphism $\;\iota_x:F_1\longrightarrow X\;$ satisfying $\;\iota_x(a)=x.\;$
    \end{definition}
    Thus, choosing an element of a biquandle $X$ determines a homomorphism $F_1\rightarrow X$.

\section{Biquandle-based invariants under connected sum}
\label{sec:connected-sum-results}
    
    In this section, we begin with a
    diagrammatic description of connected sum of virtual knotoids which is convenient for the algebraic constructions
    that follow and computations of invariant values of connected sums through the values of summands.

\subsection{Connected sum of virtual knotoids via virtual connection arcs}
\label{sec:new-connected-sum}

    In \cite{turaev2012knotoids}, Turaev provides a diagrammatic construction of the connected sum through \textit{normal} knotoid diagrams. In the virtual setting, we make use of
    the detour move and provide another diagrammatic description which is applicable to any representative diagram of virtual knotoids.

    By viewing $S^2$ as $\mathbb{R}^2\cup\{\infty\}$, Turaev calls a knotoid diagram $K$ in $S^2$ \textit{normal}, if the tail of $K$ is adjacent to the region containing $\infty$. Moreoever, he shows that every knotoid diagram in $S^2$ can be turned into a normal knotoid diagram by isotopy of $S^2$. Given two normal knotoid diagrams $K_1$ and $K_2$, their connected sum $K_1\#K_2$ can be obtained by attaching a copy of $K_2$ in a small regular neighborhood of the head of $K_1$ in a way that the tail of $K_2$ is attached to the head of $K_1$ in $\mathbb{R}^2$ \cite{turaev2012knotoids}.

    Here we use the term \textit{tail-normal (t-normal)} for normal knotoid diagrams described above. Then, we define a \textit{head-normal (h-normal)} knotoid diagram analogously as a knotoid diagram whose head is adjacent to the region containing $\infty$. It is clear by Turaev's observation that if $K_1$ is \textit{h}-normal and $K_2$ is \textit{t}-normal, then $K_1\# K_2$ can be simply formed by placing the two diagrams disjointly in $\mathbb{R}^2$, and connecting the head of $K_1$ to the tail of $K_2$ by an embedded arc in $\mathbb{R}^2$ whose interior is disjoint from $K_1$ and $K_2$.

    \begin{proposition}
    \label{prop:virtual-connecting-arc}
        Let $K_1$ and $K_2$ be two virtual knotoid diagrams. Join the head of $K_1$ to the tail of $K_2$ by an embedded arc $\alpha$, such that each intersection of the interior of $\alpha$ and $K_1\cup K_2$ is transversal and declared to be a virtual crossing. The resulting virtual knotoid diagram is equivalent to the virtual knotoid diagram $K_1\# K_2$ obtained by the usual connected sum operation, independent of the choice of $\alpha$. We call $\alpha$ a \textit{virtual connection arc} (see Figure~\ref{fig:connection arc}).
    \end{proposition}

    \begin{proof}
        Any virtual knotoid diagram $K$ can be put in \textit{h}-normal or \textit{t}-normal form by taking the corresponding endpoint and moving it to the region adjacent to $\infty$ by using a sequence of virtual endpoint moves. Suppose that $K_1$ is not $h$-normal and $K_2$ is not $t$-normal. We put these diagrams in their corresponding normal forms and let $\gamma_i$ be the curve traced by the corresponding endpoint of $K_i$ from its initial position to the final position along virtual endpoint moves. We denote the resulting diagrams $K_1'$ and $K_2'$, respectively. Clearly, $K_i\cong K_i'$ for $i=1,2$. Since $K_1'$ is $h$-normal and $K_2'$ is $t$-normal, we can form their connected sum $K_1'\# K_2'$ by an embedded arc $\gamma$ whose interior is disjoint from $K_1'\cup K_2'$. By well-definedness of the connected sum, we have $K_1\#K_2\cong K_1'\#K_2'$. Moreover, by reversing the virtual endpoint moves the concatenation $\gamma_1\ast\gamma\ast\bar{\gamma_2}$, where $\bar{\gamma_2}$ is $\gamma_2$ with the reverse orientation, forms a virtual connection arc. The independence of the choice of the virtual connection arc follows from the detour move.
    \end{proof}

    \begin{figure}[h!]
        \centering
        \begin{overpic}[width=\linewidth]{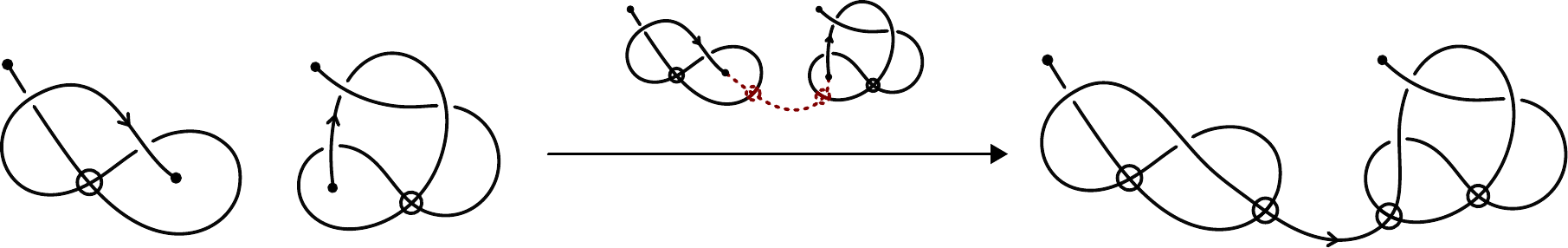}
        \put(6,-3.5){$K_1$}
        \put(24,-3.5){$K_2$}
        \put(49.5,6.9){$\alpha$}
        \put(80,-3.5){$K_1\# K_2$}
        \end{overpic}
        \vspace{0.3 cm}
        \caption{Two virtual knotoid diagrams $K_1$ and $K_2$, and their connected sum $K_1\#K_2$ obtained by using a virtual connection arc $\alpha$.}
        \label{fig:connection arc}
    \end{figure}

\subsection{Fundamental biquandles under connected sum}
    \label{sec:biquandle-behavior}

    We now study the behavior of the fundamental biquandle of virtual knotoids under connected sum categorically as a pushout in $\mathbf{Bq}$.

    Let $h_1\in\mathcal B(K_1)$ and $t_2\in\mathcal B(K_2)$ denote the generators corresponding to the head semiarc of $K_1$ and the tail semiarc of $K_2$, respectively. By the universal property of the one-generator free biquandle $F_1$, these generators determine unique homomorphisms
    \begin{align*}
        \iota_{h_1}:F_1&\rightarrow\mathcal B(K_1),&\quad \iota_{t_2}:F_1&\rightarrow\mathcal B(K_2).\\
        a&\mapsto h_1 &\quad a&\mapsto t_2
    \end{align*}
    and hence determine the following span in $\mathbf{Bq}$:
    \begin{align}
    \label{span}
        \mathcal B(K_1)\xleftarrow{\;\iota_{h_1}\;}F_1\xrightarrow{\;\iota_{t_2}\;}
        \mathcal B(K_2).
    \end{align}

    \begin{theorem}
    \label{thm:fundamental-biquandle-pushout}
        Let $K_1$ and $K_2$ be two virtual knotoid diagrams. Then
        $\mathcal B(K_1\#K_2)$ is the pushout of the span \eqref{span}
        in $\mathbf{Bq}$. That is,
        \begin{align}
        \label{isom}
        \mathcal{B}(K_1\#K_2)
        \cong
        \mathcal{B}(K_1)*_{F_1}\mathcal{B}(K_2).
        \end{align}
    \end{theorem}

    \begin{proof}
        Let
        $\;
        \mathcal B(K_1)=\langle L_1\mid R_1\rangle$, and $
        \mathcal B(K_2)=\langle L_2\mid R_2\rangle.\;$
        Consider the connected sum $K_1\# K_2$ obtained by a virtual connection arc as in Proposition~\ref{prop:virtual-connecting-arc}. Since a virtual connection arc introduces no new classical crossing, and hence no new relations, its only effect on the fundamental biquandle presentation is to
        identify the generators $h_1$ and $t_2$ corresponding to the head semiarc of $K_1$ and the tail semiarc of $K_2$, respectively. That is,
        \[
        \mathcal B(K_1\#K_2)
        \cong
        \left\langle
        L_1\sqcup L_2
        \;\middle|\;
        R_1\cup R_2\cup\{h_1=t_2\}
        \right\rangle.
        \]
        Let
        $
        \lambda_1:\mathcal B(K_1)\longrightarrow\mathcal B(K_1\#K_2)$ and $
        \lambda_2:\mathcal B(K_2)\longrightarrow\mathcal B(K_1\#K_2)$
        be the canonical homomorphisms. By construction,
        $ (\lambda_1\circ\iota_{h_1})(a) =\lambda_1(h_1) =\lambda_2(t_2) =(\lambda_2\circ\iota_{t_2})(a). $
        It follows from the defining universal property of $F_1$ that, 
        $\lambda_1\circ\iota_{h_1}=\lambda_2\circ\iota_{t_2}.$ 
        
        Now let $Y$ be a biquandle and suppose that there exist $\;f_1:\mathcal B(K_1)\longrightarrow Y$ and $f_2:\mathcal B(K_2)\longrightarrow Y\;$ such that $f_1\circ\iota_{h_1}=f_2\circ\iota_{t_2}.$ Since $F_1$ is free on one generator, this condition is equivalent to $f_1(h_1)=f_2(t_2).$ We define \(f\) on the generators of \(\mathcal B(K_1\#K_2)\) by
        $$ f(l)= \begin{cases} f_1(l),&l\in L_1,\\ f_2(l),& l\in L_2. \end{cases} $$
        Since \(f_1\) and \(f_2\) respect the relations \(R_1\) and \(R_2\), respectively, and \(f_1(h_1)=f_2(t_2)\), this assignment respects the relations of \(\mathcal B(K_1\#K_2)\). Thus, it determines a unique biquandle homomorphism
        \[
        f:\mathcal B(K_1\#K_2)\longrightarrow Y
        \]
        which satisfies $\;f\circ\lambda_1=f_1$, and $f\circ\lambda_2=f_2.\;$ Hence $\mathcal{B}(K_1\#K_2)$ satisfies the universal property of the
        pushout.
    \end{proof}

    Now we proceed by defining the pointed fundamental biquandles of virtual knotoids analogous to that in linkoid setting in \cite{pflume} and examine their behavior under connected sum in a similar way to the fundamental biquandles.

    \begin{definition}
        Let $K$ be a virtual knotoid diagram. The \textit{fundamental pointed
        biquandle} of $K$ is the $2$-pointed biquandle
        \[
        \mathcal{PB}(K)=\left(\mathcal B(K),t,h\right),
        \]
        where $t$ and $h$ are the generators of $\mathcal B(K)$ corresponding to the
        tail and head semiarcs of $K$, respectively.
    \end{definition}        

    \begin{proposition}
        The isomorphism class of $\mathcal{PB}(K)$ is an invariant of virtual knotoids.
    \end{proposition}
    
    \begin{proof}
        The fundamental biquandle is invariant under generalized Reidemeister moves. That is, if $K_1$ and $K_2$ are related by a generalized Reidemeister move, then their fundamental biquandles are isomorphic. The virtual endpoint move does not change the generator assigned to the corresponding endpoint semiarc as the labelings remain unchanged through a virtual crossing. The remaining moves are performed away from the endpoints of a virtual knotoid diagram, and hence they preserve the generators corresponding to the tail and head semiarcs of the diagram. Thus, an isomorphism between fundamental biquandles of $K_1$ and $K_2$ induces an isomorphism between their pointed fundamental biquandles.
    \end{proof}

    \begin{corollary}
        Let $K_1$ and $K_2$ be two virtual knotoid diagrams and consider their pointed fundamental biquandles $\mathcal{PB}(K_1)=\bigl(\mathcal B(K_1),t_1,h_1\bigr)$ and $\mathcal{PB}(K_2)=\bigl(\mathcal B(K_2),t_2,h_2\bigr).$ Then,
        \[
        \mathcal{PB}(K_1\#K_2)
        \cong
        \left(
        \mathcal B(K_1)*_{F_1}\mathcal B(K_2),
        \lambda_1(t_1),
        \lambda_2(h_2)
        \right),
        \]
        where $\lambda_i:\mathcal{B}(K_i)\rightarrow \mathcal{B}(K_1\# K_2)$ is the natural homomorphism.
    \end{corollary}

    \begin{proof}
        Under the isomorphism \eqref{isom} of Theorem~\ref{thm:fundamental-biquandle-pushout}, the distinguished generators corresponding to tail and
        head semiarcs of $K_1\#K_2$ in its fundamental biquandle are the images of the generators $t_1$ and
        $h_2$ under the natural homomorphisms $\lambda_1$ and $\lambda_2$, respectively. 
    \end{proof}

    Pointed fundamental biquandles also provide an interpretation of the partitioning of biquandle colorings of virtual knotoids according to the fixed endpoint semiarc colors.

    \begin{proposition}
    \label{prop:pointed-colorings}
        Let $K$ be a virtual knotoid diagram with tail and head semiarcs $t$ and $h$, respectively. If $X=\{x_1,\ldots,x_n\}$ is a biquandle, then for every $i,j\in\{1,\ldots,n\}$,
        \[
        \operatorname{Col}_X^{ij}(K)
        \cong
        \operatorname{Hom}_{\mathbf{PBq}_2}
        \bigl(
        \mathcal{PB}(K),(X,x_i,x_j)
        \bigr).
        \]
    \end{proposition}

    \begin{proof}
        A homomorphism $f:\mathcal{PB}(K)\longrightarrow (X,x_i,x_j)$ of $2$-pointed biquandles is a biquandle homomorphism $f:\mathcal B(K)\longrightarrow X$ satisfying $f(t)=x_i,\;f(h)=x_j,$ which is precisely an $X_{ij}$-coloring of $K$.
    \end{proof}

    We now turn back to the consequences of the pushout description for biquandle colorings.

    \begin{proposition}
    \label{prop:compatible-colorings}
        Let $K_1$ and $K_2$ be two virtual knotoid diagrams, and $X$ a
        biquandle. Consider the endpoint evaluation maps
        \begin{align*}
            \operatorname{ev}_h\colon \Col_X(K_1)&\longrightarrow X,&\quad \operatorname{ev}_t\colon \Col_X(K_2)&\longrightarrow X,\\
            f&\longmapsto f(h_1),&\quad f&\longmapsto f(t_2).
        \end{align*}
        Then, there is a natural bijection
        $
        \Col_X(K_1\#K_2)
        \cong
        \Col_X(K_1)\times_X\Col_X(K_2),
        $
        where the fiber product is taken with respect to
        $\operatorname{ev}_h$ and $\operatorname{ev}_t$. That is,
        \[
        \Col_X(K_1\#K_2)
        \cong
        \left\{
        (f_1,f_2)\in
        \Col_X(K_1)\times\Col_X(K_2)
        \;\middle|\;
        f_1(h_1)=f_2(t_2)
        \right\}.
        \]
    \end{proposition}

    \begin{proof}
        Recall that $\Col_X(K)=\Hom_{\mathbf{Bq}}(\mathcal B(K),X).$ For the fixed biquandle $X$, consider the functor $\Hom_{\mathbf{Bq}}(-,X):\mathbf{Bq}^{\mathrm{op}}\longrightarrow\mathbf{Set}.$
        The pushout square in Theorem~\ref{thm:fundamental-biquandle-pushout}
        \[
        \begin{tikzcd}[column sep=large,row sep=large]
        F_1
            \arrow[r,"\iota_{t_2}"]
            \arrow[d,"\iota_{h_1}"']
        &
        \mathcal B(K_2)
            \arrow[d,"\lambda_2"]
        \\
        \mathcal B(K_1)
            \arrow[r,"\lambda_1"']
        &
        \mathcal B(K_1\#K_2)
        \end{tikzcd}
        \]
        is a pullback square in $\mathbf{Bq}^{\mathrm{op}}$. Applying
        $\Hom_{\mathbf{Bq}}(-,X)$ to this square results in the pullback square
        \[
        \begin{tikzcd}[column sep=large,row sep=large]
        \Hom_{\mathbf{Bq}}(\mathcal{B}(K_1\#K_2),X)
            \arrow[r,"\lambda_1^{*}"]
            \arrow[d,"\lambda_2^{*}"']
        &
        \Hom_{\mathbf{Bq}}(\mathcal{B}(K_1),X)
            \arrow[d,"\iota_{h_1}^{*}"]
        \\
        \Hom_{\mathbf{Bq}}(\mathcal{B}(K_2),X)
            \arrow[r,"\iota_{t_2}^{*}"']
        &
        \Hom_{\mathbf{Bq}}(F_1,X).
        \end{tikzcd}
        \]
        in $\mathbf{Set}$ category. Hence we have, 
        \begin{align}
        \label{hombiject}
            \Hom_{\mathbf{Bq}}(\mathcal{B}(K_1\#K_2),X)\cong\Hom_{\mathbf{Bq}}(\mathcal{B}(K_1),X)\times_{\Hom_{\mathbf{Bq}}(F_1,X)}\Hom_{\mathbf{Bq}}(\mathcal{B}(K_2),X).
        \end{align}
        First, observe that evaluation at the single generator $a$ of $F_1$ gives the following bijection
        \[
        \begin{aligned}
        \Hom_{\mathbf{Bq}}(F_1,X)
        &\xrightarrow{\;\cong\;} X,\\
        g&\longmapsto g(a).
        \end{aligned}
        \]
        Secondly, by the compatibility condition, for two $X$-colorings $f_1\in \Hom_{\mathbf{Bq}}(\mathcal{B}(K_1),X)$ and $f_2\in \Hom_{\mathbf{Bq}}(\mathcal{B}(K_2),X)$,  we have $\iota_{h_1}^\ast(f_1)=\iota_{t_2}^\ast(f_2)$ and hence $f_1\circ\iota_{h_1}=f_2\circ\iota_{t_2}$. Since $\iota_{h_1}(a)=h_1$ and $\iota_{t_2}(a)=t_2$, we see that the compatibility condition of the pullback becomes $f_1(h_1)=f_2(t_2)$. These two observations together show that the bijection \eqref{hombiject} induces a bijection
        \[
        \Col_X(K_1\#K_2)
        \cong
        \Col_X(K_1)\times_X\Col_X(K_2),
        \]
        with the pullback compatibility condition $f_1(h_1)=f_2(t_2).$
    \end{proof}

    Let $X=\{x_1,\ldots,x_n\}$ be a biquandle. For
    $i,k\in\{1,\ldots,n\}$, Proposition~\ref{prop:compatible-colorings} implies that
    \begin{align}
    \label{eq:coloringbij}
        \Col_X^{ik}(K_1\#K_2)\cong\bigsqcup_{j=1}^{n}\left(\Col_X^{ij}(K_1)\times\operatorname{Col}_X^{jk}(K_2)\right).
    \end{align}
    Equivalently, using Proposition~\ref{prop:pointed-colorings}, we obtain
    \[
    \begin{aligned}
    \operatorname{Col}_X^{ik}(K_1\#K_2)
    \cong
    \bigsqcup_{j=1}^{n}
    \left(\operatorname{Hom}_{\mathbf{PBq}_2}
    \left(
    \mathcal{PB}(K_1),(X,x_i,x_j)
    \bigr)
    \times
    \operatorname{Hom}_{\mathbf{PBq}_2}
    \bigl(
    \mathcal{PB}(K_2),(X,x_j,x_k)
    \right)\right).
    \end{aligned}
    \]

    The following connected-sum formula for biquandle counting matrices of virtual knotoids was obtained in \cite{gugumcu2026biquandle}, and it follows mainly from the bijection \eqref{eq:coloringbij}. Hence, Proposition~\ref{prop:compatible-colorings} provides a categorical explanation of this formula.

    \begin{corollary}
        Let $K_1$ and $K_2$ be two virtual knotoid diagrams and
        $X=\{x_1,\ldots,x_n\}$ a finite biquandle. Then
        \[
        \bcm{K_1\#K_2}=\bcm{K_1}\bcm{K_2}.
        \]
    \end{corollary}
    
    \begin{proof}
        For fixed $i,k$, by Proposition~\ref{prop:compatible-colorings}, we have
        \[
        \operatorname{Col}_X^{ik}(K_1\#K_2)
        \cong
        \bigsqcup_{j=1}^{n}
        \left(
        \operatorname{Col}_X^{ij}(K_1)
        \times
        \operatorname{Col}_X^{jk}(K_2)
        \right).
        \]
        Taking cardinalities of both sides gives
        \[
        \begin{aligned}
        m_{ik}(K_1\#K_2)
        &=
        \sum_{j=1}^{n}
        \left|
        \operatorname{Col}_X^{ij}(K_1)
        \right|
        \left|
        \operatorname{Col}_X^{jk}(K_2)
        \right| =
        \sum_{j=1}^{n}
        m_{ij}(K_1)m_{jk}(K_2),
        \end{aligned}
        \]
        which is the $(i,k)$-entry of
        $\bcm{K_1}\bcm{K_2}$.
    \end{proof}

\subsection{Biquandle virtual bracket invariants under connected sum}
\label{sec:bracket-behavior}

    In this section, we study the behavior of \textit{normalized} biquandle virtual bracket invariants of virtual knotoids under connected sum.
\subsubsection{The normalized biquandle virtual bracket invariant}

    The Kauffman bracket has two closely related state-sum formulations. These formulations differ in the component factor assigned to each state $s$. Tne first one uses $\delta^{|s|}$, as in the formulation given in Example~\ref{ex:kaufmann}, while the other one uses $\delta^{|s|-1}$ in the formula. Turaev showed in~\cite{turaev2012knotoids} that
    the Kauffman bracket of knotoids is multiplicative under connected sum when the latter convention is used. Since every state of a virtual knotoid diagram contains a unique open component, the exponent $|s|-1$ counts exactly the number of closed components in $s$.

    Motivated by this observation, we introduce a one-component-reduced
    normalization of the biquandle virtual bracket state-sum.

    \begin{definition}
        Let $K$ be a virtual knotoid diagram, $X$ a biquandle, and $\beta$ a biquandle virtual bracket. For an $X$-coloring $f\in\operatorname{Col}_X(K),$
        the \textit{normalized biquandle virtual bracket value} of $K_f$ is the value obtained by the state-sum formula
        \[
        \bvbn{K_f}=\omega^{-\operatorname{wr}(K)}\sum_{s\in\mathcal{S}(K)}\left(\prod_{v\in C(K)} c_{s,f}(v)
        \right)\delta^{|s|-1}.
        \]
    \end{definition}

    \begin{remark}
    \label{rem:reduced-vs-ordinary}
        For every $X$-colored virtual knotoid diagram $K_f$, it is clear that the ordinary and normalized biquandle virtual bracket values are related by $\bvb{K_f}=\delta\,\bvbn{K_f}.$
        Thus, the ordinary biquandle virtual bracket value can be recovered from the normalized value by multiplication by $\delta$. If $\delta$ is invertible, then the normalized value can be recovered $\bvbn{K_f}=\delta^{-1}\,\bvb{K_f}.$ However, $\delta$ does not have to be invertible for biquandle virtual brackets in general.
    \end{remark}

    \begin{proposition}
    \label{prop:reduced-bvb-invariant}
        The normalized biquandle virtual bracket value $\bvbn{K_f}$ is invariant under generalized Reidemeister moves of biquandle colored virtual knotoid diagrams.
    \end{proposition}

    \begin{proof}
        The invariance of the biquandle virtual bracket value for virtual links follows from the choice of axioms in the construction. These axioms are motivated by local state contribution comparisons before and after a Reidemeister move \cite{nelson2019biquandle}, and the same argument applies to biquandle colored virtual knotoid diagrams \cite{gugumcu2026biquandle}. Hence, the same axiom set of biquandle virtual brackets imply invariance when we replace the factor $\delta^{|s|}$ by
        $\delta^{|s|-1}$ in the state-sum formula.
    \end{proof}

    Recall from Section~\ref{sec:biquandle-prelims} that the fundamental biquandle virtual bracket expression $\fbvb{K}$ provides a universal expression that could be used to compute the fundamental virtual bracket value by specialization and evaluation. The same approach applies in the normalized setting. 

    For a virtual knotoid diagram $K$, let $\fbvbn{K}$ denote the \textit{normalized fundamental biquandle virtual bracket expression} of $K$, obtained by labeling each semiarc by its corresponding generator of $\mathcal B(K)$ and forming the formal state sum
    \[
    \fbvbn{K} = \omega^{-\operatorname{wr}(K)} \sum_{s\in\mathcal{S}(K)}
    \left(\prod_{v\in C(K)}c_{s,1_{\mathcal{B}(K)}}(v)\right)\delta^{|s|-1}.
    \]
    As in the unnormalized setting, for $f\in\Col_X(K)$, evaluating $\fbvbn{K}$ with respect to $f$ and the chosen biquandle virtual bracket $\beta$ gives $\bvbn{K_f}$. Hence, starting from the behavior of this expression under connected sum is key in studying the behavior of invariants based on biquandle virtual brackets.

    Let $\varphi: Y\rightarrow Z$ be a biquandle homomorphism and let $\varphi^*$ denote the relabeling of a formal biquandle virtual bracket expression over $Y$ obtained by replacing $F_{a,b}\longmapsto F_{\varphi(a),\varphi(b)},$ where $F\in\{A,B,V,C,D,U\},$ and leaving $\delta$ and $\omega$ unchanged.

    \begin{theorem}
    \label{thm:fundamental-bvb-product}
        Let $K_1$ and $K_2$ be two virtual knotoid diagrams, and let $\lambda_i:\mathcal B(K_i)\longrightarrow\mathcal B(K_1\#K_2),\;i=1,2,$ be the canonical homomorphisms arising from the pushout of
        Theorem~\ref{thm:fundamental-biquandle-pushout}. Then
        \begin{align}
        \label{eq:fbvbnidentity}
            \fbvbn{K_1\#K_2}=\lambda_1^\ast
            \left(\fbvbn{K_1}\right)\,\lambda_2^\ast
            \left(\fbvbn{K_2}\right).
        \end{align}
    \end{theorem}

    \begin{proof}
        Every state $s$ of $K_1\#K_2$ is obtained uniquely from a pair $(s_1,s_2)\in
        \mathcal S(K_1)\times\mathcal S(K_2)$ by joining the open component of $s_1$ to the open component of $s_2$. Thus, there is a natural bijection
        \[
        \mathcal S(K_1\#K_2)\cong
        \mathcal S(K_1)\times\mathcal S(K_2).
        \]
        Under this correspondence, all closed components of $s_1$ and $s_2$ remain closed. Hence,
        \[
        |s|-1=(|s_1|-1)+(|s_2|-1).
        \]
        Moreover, the crossing set of $K_1\#K_2$ is the union of crossing sets of $K_1$ and $K_2$. Hence, $\operatorname{wr}(K_1\#K_2)=
        \operatorname{wr}(K_1)+\operatorname{wr}(K_2),$ and the smoothing coefficients of $s$ split into the coefficients coming from the crossings of $K_1$ and those coming from the crossings of $K_2$.
        
        The identity coloring $1_{\mathcal{B}(K_1\#K_2)}$ of $K_1\#K_2$ restricts to the coloring of $K_i$ induced by
        $\lambda_i$. Hence, the smoothing coefficients associated to the crossings of $K_i$ are precisely the corresponding coefficients of $s_i$ relabeled by $\lambda_i$. Therefore, for the state $s$ corresponding to $(s_1,s_2)$, the contribution of $s$ is the product of the relabeled contributions of $s_1$ and $s_2$. Summing over $\mathcal S(K_1)\times\mathcal S(K_2)$
        gives
        \[
        \fbvbn{K_1\#K_2}=\lambda_1^\ast
        \left(\fbvbn{K_1}\right)\,\lambda_2^\ast
        \left(\fbvbn{K_2}\right).
        \]
    \end{proof}

    \begin{example}
        Figure~\ref{fig:consumstate} illustrates the correspondence between a pair of states $(s_1,s_2)\in S(K_1)\times S(K_2)$ and the state $s=s_1\#s_2\in\mathcal S(K_1\#K_2)$, together with the factorization of the associated normalized state weight. To simplify notations in the figure, for a generator $x$ we let 
        \[
            x^\ast=\begin{cases}
                \lambda_1(x),& x\in\mathcal{B}(K_1),\\
                \lambda_2(x),& x\in\mathcal{B}(K_2).
            \end{cases}
        \]
        Without the writhe normalization, the normalized state weights of $s_1$ and $s_2$ are $V_{t,q}C_{s,p}\delta$ and $D_{d,f}C_{f,a}C_{c,d}\delta$, respectively. The corresponding state $s=s_1\#s_2$ has normalized state weight
        \[
        V_{\lambda_1(t),\lambda_1(q)}
        C_{\lambda_1(s),\lambda_1(p)}
        D_{\lambda_2(d),\lambda_2(f)}
        C_{\lambda_2(f),\lambda_2(a)}
        C_{\lambda_2(c),\lambda_2(d)}
        \delta^2,
        \]
        which is the product of the two state weights after relabeling the generators
        by $\lambda_1$ and $\lambda_2$, as expected.
        \vspace{1 cm}
        \begin{figure}[h!]
            \centering
            \begin{overpic}[width=0.95\linewidth]{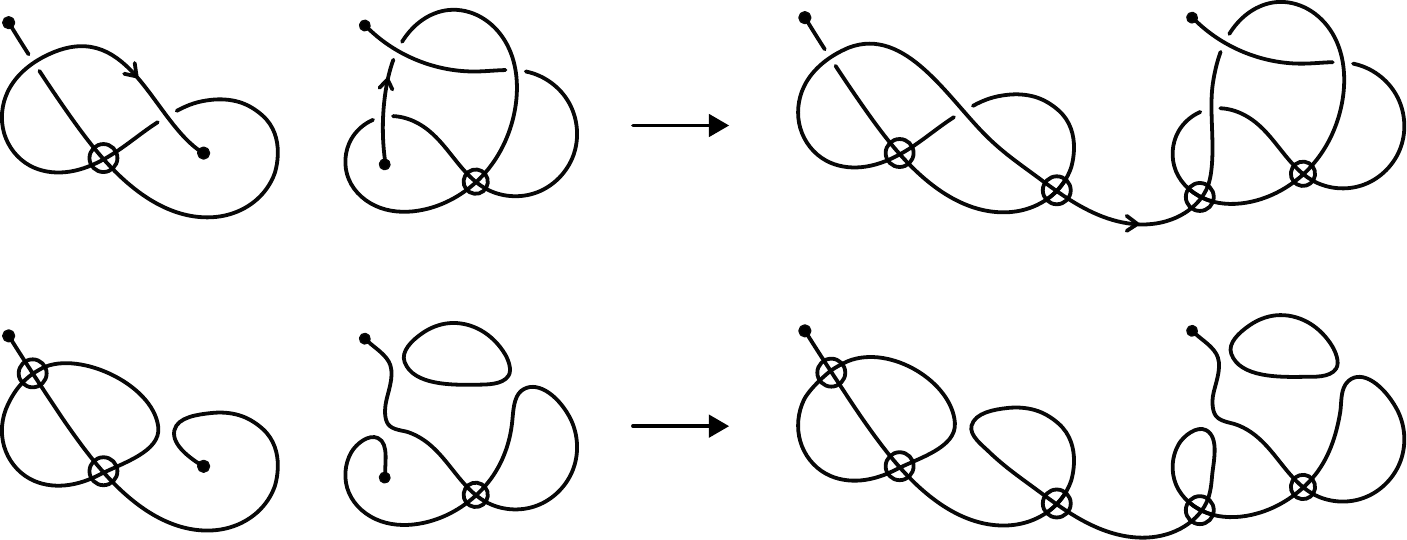}
            \put(-1.5,35){\footnotesize $t$}
            \put(11,21){\footnotesize $s$}
            \put(1,24){\footnotesize $r$}
            \put(7,36){\footnotesize $q$}
            \put(11.5,26){\footnotesize $p$}
        
            \put(28.5,26){\footnotesize $g$}
            \put(25,32){\footnotesize $f$}
            \put(31,39){\footnotesize $e$}
            \put(37,28){\footnotesize $d$}
            \put(37,21.5){\footnotesize $c$}
            \put(31,31){\footnotesize $b$}
            \put(23.5,36){\footnotesize $a$}
        
            \put(54,35){\footnotesize $t^\ast$}
            \put(67.5,21){\footnotesize $s^\ast$}
            \put(57,24){\footnotesize $r^\ast$}
            \put(63,36){\footnotesize $q^\ast$}
            \put(75,20){\footnotesize $p^\ast=g^\ast$}
        
            \put(83.5,32.5){\footnotesize $f^\ast$}
            \put(90,39.5){\footnotesize $e^\ast$}
            \put(96,28.5){\footnotesize $d^\ast$}
            \put(95,22){\footnotesize $c^\ast$}
            \put(90,31.5){\footnotesize $b^\ast$}
            \put(81.5,36.5){\footnotesize $a^\ast$}
        
            \put(-3.5,11){\footnotesize $V_{t,q}$}
            \put(10.5,10.5){\footnotesize $C_{s,p}$}
            \put(22,8.5){\footnotesize $D_{d,f}$}
            \put(25.5,15.5){\footnotesize $C_{f,a}$}
            \put(37,12){\footnotesize $C_{c,d}$}
        
            \put(52,11.5){\footnotesize $V_{t^\ast,q^\ast}$}
            \put(68,10.5){\footnotesize $C_{s^\ast,p^\ast}$}
            \put(79,9){\footnotesize $D_{d^\ast,f^\ast}$}
            \put(83.5,16.5){\footnotesize $C_{f^\ast,a^\ast}$}
            \put(96,12.5){\footnotesize $C_{c^\ast,d^\ast}$}
        
            \put(7,43){$K_1$}
            \put(31,43){$K_2$}
            \put(70,43){$K=K_1\#K_2$}
            \put(7,-3){$s_1$}
            \put(31,-3){$s_2$}
            \put(73,-3){$s=s_1\#s_2$}
            \end{overpic}
            \vspace{0.5 cm}
            \caption{
            A pair of states $s_1\in\mathcal S(K_1)$ and $s_2\in\mathcal S(K_2)$
            determines a state $s\in\mathcal S(K)$.
            The starred labels on $K$ and $s$ denote the images of the corresponding semiarc generators under the canonical homomorphisms $\lambda_1\colon \mathcal B(K_1)\to\mathcal B(K_1\#K_2)$ and
            $\lambda_2\colon \mathcal{B}(K_2)\to\mathcal B(K_1\#K_2)$. The relation $p^\ast=g^\ast$ follows from the identification along the virtual connection arc.}
            \label{fig:consumstate}
        \end{figure}
    \end{example}

    \begin{corollary}
    \label{cor:reduced-bvb-product}
        Let $X$ be a biquandle, $\beta$ a biquandle virtual bracket, and let $K_1$ and $K_2$ be two virtual
        knotoid diagrams. Suppose that
        $f_1\in\Col_X(K_1)$ and $f_2\in\Col_X(K_2)$ are compatible, that is, $f_1(h_1)=f_2(t_2).$
        Let $f=f_1\# f_2\in\Col_X(K_1\#K_2)$ denote the unique coloring corresponding to $(f_1,f_2)$. Then,
        \[
        \bvbn{(K_1\#K_2)_{f}}=
        \bvbn{(K_1)_{f_1}}
        \bvbn{(K_2)_{f_2}}.
        \]
    \end{corollary}

    \begin{proof}
        Evaluating the identity \eqref{eq:fbvbnidentity} of
        Theorem~\ref{thm:fundamental-bvb-product} with respect to the coloring $f$ and the bracket $\beta$ gives
        \[
        \bvbn{(K_1\#K_2)_{f}}=
        \bvbn{(K_1)_{f\circ \lambda_1}}
        \bvbn{(K_2)_{f\circ \lambda_2}}.
        \]
        By definition of the coloring $f=f_1\# f_2$, we have $f\circ\lambda_1=f_1,$ and $f\circ\lambda_2=f_2.$
    \end{proof}

    \begin{corollary}
        Under the hypotheses of
        Corollary~\ref{cor:reduced-bvb-product},
        \[
        \bvb{(K_1)_{f_1}}
        \bvb{(K_2)_{f_2}}=\delta\,
        \bvb{(K_1\#K_2)_{f}}.
        \]
        In particular, if $\delta$ is invertible in $R$, then we have a way of obtaining the unnormalized bracket value of the connected sum from the values of the summands:
        \[
        \bvb{(K_1\#K_2)_{f}}=\delta^{-1}\,\bvb{(K_1)_{f_1}}
        \bvb{(K_2)_{f_2}}.
        \]
    \end{corollary}

    \begin{proof}
        By Remark~\ref{rem:reduced-vs-ordinary} and
        Corollary~\ref{cor:reduced-bvb-product},
        \[
        \begin{aligned}
        \bvb{(K_1)_{f_1}}
        \bvb{(K_2)_{f_2}}
        &=
        \delta^2
        \bvbn{(K_1)_{f_1}}
        \bvbn{(K_2)_{f_2}}\\
        &=
        \delta^2\bvbn{(K_1\#K_2)_{f}}\\
        &=
        \delta\,
        \bvb{(K_1\#K_2)_{f}}.
        \end{aligned}
        \]
        If $\delta$ is invertible, the second formula follows by multiplying both sides
        by $\delta^{-1}$.
    \end{proof}

\subsubsection{Normalized biquandle virtual bracket multiset invariants}

    In this subsection, we study the normalized analogs of the biquandle virtual bracket multiset invariants of virtual knotoids reviewed in Section~\ref{sec:biquandle-prelims}. 

    \begin{definition}
        Let $K$ be a virtual knotoid diagram, $X$ a finite biquandle, and let
        $\beta$ be a biquandle virtual bracket. The \textit{normalized biquandle virtual bracket multiset} of $K$ is
        \[
        \bvbnmult{K}=
        \left\{\left\{
        \bvbn{K_f}
        \;\middle|\;
        f\in\operatorname{Col}_X(K)
        \right\}\right\}.
        \]
        For $x_i,x_j\in X$, the \emph{$(i,j)$-normalized biquandle virtual bracket
        multiset} is
        \[
        \widetilde{\Phi}_{X_{ij}}^{\beta,M}(K)
        =
        \left\{\left\{
        \bvbn{K_f}
        \;|\;
        f\in\operatorname{Col}_X^{ij}(K)
        \right\}\right\}.
        \]
        Finally, the family 
        \[\left\{
        \widetilde{\Phi}_{X_{ij}}^{\beta,M}(K)
        \right\}_{i,j=1}^{n}
        \]
        is the \textit{normalized biquandle virtual bracket multiset family}.
    \end{definition}

    \begin{remark}
    \label{rem:comparison}
        Let $\beta$ be a biquandle virtual bracket over a commutative and unitary ring $R$. Consider the map
        \begin{align*}
            m_\delta:R&\longrightarrow R,\\
            r&\longmapsto \delta r.
        \end{align*}
        The ordinary biquandle virtual bracket multiset
        is obtained from its normalized version by
        applying $m_\delta$ to each element, preserving
        multiplicities. Thus, the normalized multiset
        invariant is at least as strong as its unnormalized version. If $\delta$ is not a zero divisor in $R$, then
        $m_\delta$ is injective, and the two invariants
        have the same distinguishing power. This is the
        case for the Kauffman bracket over
        $R=\mathbb{Z}[A,A^{-1}]$, where
        $\delta=-A^2-A^{-2}$ is
        not a zero divisor. However, if $\delta$ is a zero divisor,
        then $m_\delta$ is not injective. Hence, in this case, multiplication by $\delta$ causes loss of information and consequently,
        the normalized multiset invariant is
        strictly stronger than its unnormalized version.
        We demonstrate this by an explicit
        example in Section~\ref{sec:bvb-matrices}.
    \end{remark}

    As in the unnormalized setting, it is clear that we have 
    \[
        \bvbnmult{K})=
        \biguplus_{i,j=1}^{n}
        \widetilde{\Phi}_{X_{ij}}^{\beta,M}(K).
    \]
    
    It follows from the invariance of normalized biquandle virtual bracket value under generalized Reidemeister moves that these multisets are invariants of virtual knotoids. To be able to describe the behavior of these multiset invariants under connected sum, we
    fix our notation for multiplying multisets.

    \begin{definition}
        Let $\mathcal{A}$ and $\mathcal{B}$ be finite multisets with elements in the
        commutative ring $R$. The \textit{product} of these multisets is
        \[
        \mathcal{A}\cdot\mathcal{B}=\left\{\left\{
        ab\;|\;a\in\mathcal{A},\ b\in\mathcal{B}\right\}\right\},
        \]
        where the multiplicity of an element $r\in R$ is the number of pairs
        $(a,b)$, counted with their multiplicity, such that $ab=r$.
    \end{definition}

    The multiplicative behavior of the normalized biquandle virtual bracket value for compatible biquandle colorings provides a composition rule for the family of normalized biquandle virtual bracket multisets.

    \begin{theorem}
    \label{thm:reduced-bvb-multiset-product}
        Let $X=\{x_1,\ldots,x_n\}$ be a biquandle, $\beta$ a
        biquandle virtual bracket, and let
        $K_1$ and $K_2$ be two virtual knotoid diagrams. Then, for every
        $i,k\in\{1,\ldots,n\}$,
        \[
        \widetilde{\Phi}_{X_{ik}}^{\beta,M}(K_1\#K_2)
        =
        \biguplus_{j=1}^{n}
        \left(
        \widetilde{\Phi}_{X_{ij}}^{\beta,M}(K_1)
        \cdot
        \widetilde{\Phi}_{X_{jk}}^{\beta,M}(K_2)
        \right).
        \]
    \end{theorem}

    \begin{proof}
        By Proposition~\ref{prop:compatible-colorings}, there is a bijection
        \[
        \operatorname{Col}_X^{ik}(K_1\#K_2)
        \cong
        \bigsqcup_{j=1}^{n}
        \left(
        \operatorname{Col}_X^{ij}(K_1)
        \times
        \operatorname{Col}_X^{jk}(K_2)
        \right).
        \]
        Let $f$ be the coloring of $K_1\#K_2$ corresponding to the coloring pair
        $(f_1,f_2)\in\operatorname{Col}_X^{ij}(K_1)
        \times\operatorname{Col}_X^{jk}(K_2)$. By Corollary~\ref{cor:reduced-bvb-product}, the normalized biquandle virtual
        bracket value associated with this coloring satisfies
        \[
        \bvbn{(K_1\#K_2)_{f}}=
        \bvbn{(K_1)_{f_1}}\bvbn{(K_2)_{f_2}}.
        \]
        Hence, for a fixed intermediate color $x_j$, the multiset of values resulting
        from compatible pairs of colorings is precisely the multiset product
        \[
        \widetilde{\Phi}_{X_{ij}}^{\beta,M}(K_1)
        \cdot
        \widetilde{\Phi}_{X_{jk}}^{\beta,M}(K_2).
        \]
        Taking the multiset union over all possible intermediate colors $x_j$ for $j=1,\dots ,n$ completes the proof.
    \end{proof}

    \begin{corollary}
    \label{cor:reduced-bvb-total-multiset-product}
        Under the hypotheses of
        Theorem~\ref{thm:reduced-bvb-multiset-product},
        \[
        \bvbnmult{K_1\#K_2}
        =\biguplus_{i,j,k=1}^{n}\left(
        \widetilde{\Phi}_{X_{ij}}^{\beta,M}(K_1)
        \cdot\widetilde{\Phi}_{X_{jk}}^{\beta,M}(K_2)
        \right).
        \]
    \end{corollary}

    \begin{proof}
        The biquandle virtual bracket multiset of $K_1\#K_2$ satisfies
        \[
        \bvbnmult{K_1\#K_2}
        =
        \biguplus_{i,k=1}^{n}
        \widetilde{\Phi}_{X_{ik}}^{\beta,M}(K_1\#K_2).
        \]
        Applying Theorem~\ref{thm:reduced-bvb-multiset-product} to each term on the right-hand side gives
        \[
        \bvbnmult{K_1\#K_2}=
        \biguplus_{i,k=1}^{n}\biguplus_{j=1}^{n}
        \left(\widetilde{\Phi}_{X_{ij}}^{\beta,M}(K_1)\cdot
        \widetilde{\Phi}_{X_{jk}}^{\beta,M}(K_2)
        \right)=\biguplus_{i,j,k=1}^{n}
        \left(\widetilde{\Phi}_{X_{ij}}^{\beta,M}(K_1)\cdot
        \widetilde{\Phi}_{X_{jk}}^{\beta,M}(K_2)\right).
        \]
    \end{proof}

    \subsubsection{Normalized biquandle virtual bracket matrices}
    \label{sec:bvb-matrices}
    
    We now specialize our results to the case where the coefficient ring $R$ of a biquandle virtual bracket is a number ring. In this case, as in Section~\ref{sec:biquandle-prelims}, the multisets of normalized biquandle virtual bracket values can be encoded by polynomials.

    \begin{definition}
        Let $K$ be a virtual knotoid diagram, $X$ be a finite biquandle, and let
        $\beta$ be a biquandle virtual bracket over a commutative unitary number ring $R$.
        The \textit{normalized biquandle virtual bracket polynomial} of $K$ is defined by
        \[
        \widetilde{\Phi}_X^{\beta}(K)=
        \sum_{f\in\Col_X(K)}
        u^{\bvbn{K_f}},
        \]
        where $u$ is a formal variable.
    \end{definition}

    For fixed elements $x_i,x_j\in X$, we define
    \[
    \widetilde{\Phi}_{X_{ij}}^{\beta}(K)=
    \sum_{f\in\operatorname{Col}_X^{ij}(K)}
    u^{\bvbn{K_f}}.
    \]
    That is, $\widetilde{\Phi}_{X_{ij}}^{\beta}(K)$ is the polynomial encoding the
    multiset $\widetilde{\Phi}_{X_{ij}}^{\beta,M}(K).$
    Since the set of all $X$-colorings of $K$ is the disjoint union of its
    endpoint-partitioned coloring sets, we have
    \[
    \widetilde{\Phi}_X^{\beta}(K)=\sum_{i,j=1}^{n}
    \widetilde{\Phi}_{X_{ij}}^{\beta}(K).
    \]

    The multiset composition rule of
    Theorem~\ref{thm:reduced-bvb-multiset-product} involves multiplication of biquandle virtual bracket values. If two bracket values are $a$ and $b$ for $K_1$ and $K_2$, respectively, then the contribution to the value of the connected sum $K_1\#K_2$ is $ab$. Hence, the corresponding monomial in the biquandle virtual bracket polynomial is $u^{ab}$. Thus, we use the following product.

    \begin{definition}
        Let $p(u)=\sum_{a}m_a u^a$ and $
        q(u)=\sum_{b}n_b u^b$ be polynomials encoding finite multisets of elements of $R$. We define their
        \textit{bracket product} by
        \[
        p(u)\star q(u)=\sum_{a,b}m_a n_b\,u^{ab}.
        \]
    \end{definition}

    \begin{remark}
    \label{rem:encodings}
        The operation $\star$ encodes the multiset
        product. That is, if $p(u)$ and $q(u)$ encode multisets $\mathcal{A}$ and $\mathcal{B}$, respectively, then $p(u)\star q(u)$ encodes their multiset product $\mathcal{A}\cdot\mathcal{B}.$
        Similarly, ordinary polynomial addition encodes multiset union with additive multiplicities.
    \end{remark} 
    
    We can now describe the behavior of the normalized biquandle virtual bracket
    polynomial under connected sum.

    \begin{corollary}
        Let $X=\{x_1,\ldots,x_n\}$ be a biquandle, $\beta$ a
        biquandle virtual bracket over a commutative unitary number ring $R$, and
        let $K_1$ and $K_2$ be two virtual knotoid diagrams. Then
        \begin{align}
        \label{eq:multisets}
            \widetilde{\Phi}_X^{\beta}(K_1\#K_2)=
            \sum_{i,j,k=1}^{n}\left(
            \widetilde{\Phi}_{X_{ij}}^{\beta}(K_1)\star
            \widetilde{\Phi}_{X_{ij}}^{\beta}(K_2)\right).
        \end{align}
    \end{corollary}

    \begin{proof}
        By Corollary~\ref{cor:reduced-bvb-total-multiset-product}, we have
        \begin{align*}
            \bvbnmult{K_1\#K_2}=
            \biguplus_{i,j,k=1}^{n}\left(
            \widetilde{\Phi}_{X_{ij}}^{\beta,M}(K_1)
            \cdot\widetilde{\Phi}_{X_{jk}}^{\beta,M}(K_2)
            \right).
        \end{align*}
        Each multiset in this equation is encoded by its corresponding polynomial. Hence, by Remark~\ref{rem:encodings}, we obtain \eqref{eq:multisets}.
    \end{proof}

    We define the normalized biquandle virtual bracket matrix analogously to the unnormalized version in Section~\ref{sec:biquandle-prelims}.

    \begin{definition}
        Let $X=\{x_1,\ldots,x_n\}$ be a biquandle, $K$ a virtual
        knotoid diagram, and $\beta$ a biquandle virtual bracket over a
        commutative number ring $R$. The \textit{normalized biquandle virtual bracket
        matrix} of $K$ is
        \[
        \widetilde{\mathcal{M}}_X^{\beta}(K)=
        \begin{bmatrix}
        \widetilde{\Phi}_{X_{11}}^{\beta}(K)&\cdots&
        \widetilde{\Phi}_{X_{1n}}^{\beta}(K)\\
        \vdots&\ddots&\vdots\\
        \widetilde{\Phi}_{X_{n1}}^{\beta}(K)&\cdots&
        \widetilde{\Phi}_{X_{nn}}^{\beta}(K)
        \end{bmatrix}.
        \]
    \end{definition}

    We now provide an example demonstrating that the normalized biquandle virtual bracket multiset invariant can be strictly stronger than its unnormalized version, as mentioned in Remark~\ref{rem:comparison}.

    \begin{example}
        Consider the biquandle $X=\mathbb{Z}_3$ with operations defined as
        \[
        x \un y = x+y+1, \qquad
        x \ovr y = 2x+1
        \]
        for all $x,y\in X$. Our Python and Mathematica computations show that the following matrix of coefficient maps defines a biquandle virtual bracket $\beta$
        over $R=\mathbb{Z}_9$:
        \[
        \left[
        \begin{array}{ccc|ccc|ccc|ccc|ccc|ccc}
        6 & 6 & 3 & 6 & 6 & 3 & 2 & 5 & 7 & 3 & 3 & 6 & 3 & 3 & 6 & 5 & 2 & 4 \\
        3 & 6 & 6 & 3 & 6 & 6 & 4 & 2 & 2 & 6 & 3 & 3 & 6 & 3 & 3 & 7 & 5 & 5 \\
        6 & 6 & 6 & 6 & 6 & 6 & 8 & 8 & 2 & 3 & 3 & 3 & 3 & 3 & 3 & 8 & 8 & 5
        \end{array}
        \right],
        \]
        with $\delta=3$ and $w=8$ where the blocks of the matrix represent the operation tables of the coefficient maps $A,B,V,C,D$ and $U,$ respectively.

        \begin{figure}[h!]
		\begin{subfigure}{0.5\textwidth}
			\centering
			\begin{overpic}[width=30 mm]{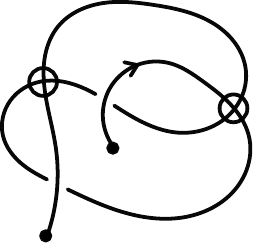}
            \put(31,30){$a_1$}
            \put(63,15){$a_2$}
            \put(-13,44){$a_3$}
            \put(33,82){$a_4$}
            \put(2,3){$a_5$}
            \end{overpic}
			\caption{A virtual knotoid diagram $K_1$.}
            \vspace{0.3 cm}
			\label{21}
		\end{subfigure}
		\begin{subfigure}{0.5\textwidth}
			\centering
			\begin{overpic}[width=35 mm]{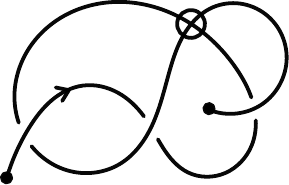}
            \put(-10,5){$a_1$}
            \put(25,25){$a_2$}
            \put(85,1){$a_3$}
            \put(35,54){$a_4$}
            \put(25,-4){$a_5$}
            \put(97,58){$a_6$}
            \put(68,16){$a_7$}
            \end{overpic}
            \vspace{0.3 cm}
			\caption{A virtual knotoid diagram $K_2$.}
            \vspace{0.3 cm}
			\label{31}
		\end{subfigure}
        \begin{subfigure}{0.5\textwidth}
			\centering
			\begin{overpic}[width=60 mm]{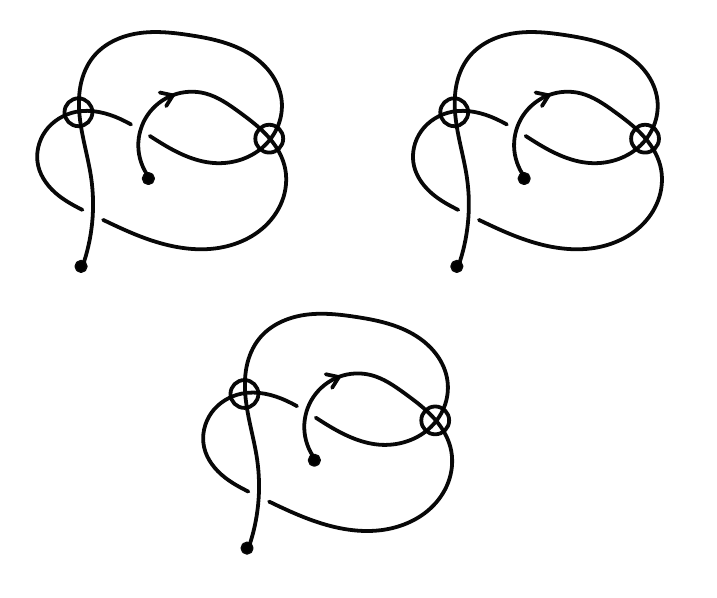}
                \put(39.5,16){\small 3}
                \put(52,11){\small 1}
                \put(24,20){\small 2}
                \put(41,34){\small 1}
		        \put(30,7){\small 3}

                \put(69,55){\small 1}
                \put(81.5,50.5){\small 3}
                \put(53.5,59){\small 2}
                \put(70.5,73){\small 3}
		        \put(59.5,46){\small 1}

                \put(16.5,55){\small 2}
                \put(29,50.5){\small 2}
                \put(1,59){\small 2}
                \put(18,73){\small 2}
		        \put(7,46){\small 2}
			\end{overpic}
			\caption{$\mathbb{Z}_3$-colorings of $K_1$.}
			\label{21a}
		\end{subfigure}
        \begin{subfigure}{0.5\textwidth}
			\centering
			\begin{overpic}[width=65 mm]{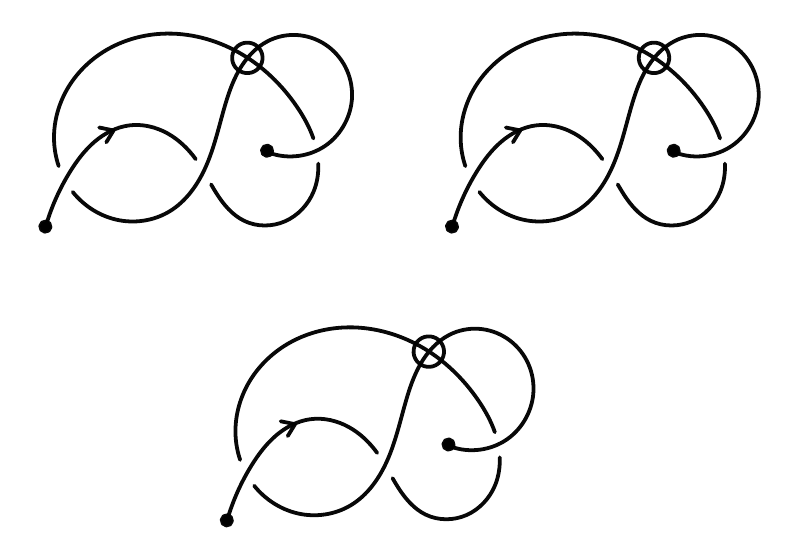}
                \put(38,12.5){\small 1}
                \put(59,1){\small 3}
                \put(37,0){\small 3}
                \put(41,24){\small 1}
		        \put(24,3){\small 3}
                \put(66,26){\small 1}
                \put(55,8.5){\small 3}

                \put(65,49){\small 3}
                \put(87,37.5){\small 1}
                \put(65,36.5){\small 1}
                \put(69,60.5){\small 3}
		        \put(52,39.5){\small 1}
                \put(94,62.5){\small 3}
                \put(83,45){\small 1}

                \put(14.5,49){\small 2}
                \put(36.5,37.5){\small 2}
                \put(14.5,36.5){\small 2}
                \put(18.5,60.5){\small 2}
		        \put(1.5,39.5){\small 2}
                \put(43.5,62.5){\small 2}
                \put(32.5,45){\small 2}
			\end{overpic}
            \vspace{0.2 cm}
			\caption{ $\mathbb{Z}_3$-colorings of $K_2$.}
			\label{31a}
		\end{subfigure}
		\caption{$\mathbb{Z}_3$-colorings of two virtual knotoid diagrams.}
		\label{2242}
	\end{figure}
    Two virtual knotoid diagrams $K_1$ and $K_2$ are given in Figure~\ref{21} and \ref{31}, respectively. These two diagrams admit three $\Z_3$-colorings each, as depicted in Figures~\ref{21a} and \ref{31a}, respectively. Biquandle virtual matrices of $K_1$ and $K_2$ are identical and equal to the identity matrix.
    
    The biquandle virtual bracket expression of $K_1$ is explicitly computed in Example 4.9 of \cite{gugumcu2026biquandle} and it reads:
    \begin{align*}
		&\fbvb{K_1}=\omega^{2}\,[\,\delta^2(D_{a_4,a_1}D_{a_3,a_4}+C_{a_4,a_1}U_{a_3,a_4}+U_{a_4,a_1}C_{a_3,a_4})\notag\\&+\delta(D_{a_4,a_1}C_{a_3,a_4}+D_{a_4,a_1}U_{a_3,a_4}+C_{a_4,a_1}D_{a_3,a_4}+C_{a_4,a_1}C_{a_3,a_4}+U_{a_4,a_1}D_{a_3,a_4}+U_{a_4,a_1}U_{a_3,a_4})].
	\end{align*}
    Hence, its normalized version is,
    \begin{align*}
		&\fbvbn{K_1}=\omega^{2}\,[\delta(D_{a_4,a_1}D_{a_3,a_4}+C_{a_4,a_1}U_{a_3,a_4}+U_{a_4,a_1}C_{a_3,a_4})\notag\\&+D_{a_4,a_1}C_{a_3,a_4}+D_{a_4,a_1}U_{a_3,a_4}+C_{a_4,a_1}D_{a_3,a_4}+C_{a_4,a_1}C_{a_3,a_4}+U_{a_4,a_1}D_{a_3,a_4}+U_{a_4,a_1}U_{a_3,a_4}].
	\end{align*}
   
The virtual knotoid diagram $K_2$ has three classical crossings, thus it has $3^3=27$ states. Summing the contributions of all these states, we obtain the normalized biquandle virtual bracket expression of $K_2$ as
\begin{align*}
\fbvbn{K_2}=\omega^3\Big[&
\delta^2\big(D_{a_5,a_1}D_{a_3,a_5}D_{a_4,a_6}\big)\\
&+\delta\big(
U_{a_5,a_1}U_{a_3,a_5}C_{a_4,a_6}
+U_{a_5,a_1}C_{a_3,a_5}U_{a_4,a_6}
+U_{a_5,a_1}D_{a_3,a_5}D_{a_4,a_6}\\
&\qquad+C_{a_5,a_1}U_{a_3,a_5}U_{a_4,a_6}
+C_{a_5,a_1}C_{a_3,a_5}C_{a_4,a_6}
+C_{a_5,a_1}D_{a_3,a_5}D_{a_4,a_6}\\
&\qquad+D_{a_5,a_1}U_{a_3,a_5}D_{a_4,a_6}
+D_{a_5,a_1}C_{a_3,a_5}D_{a_4,a_6}
+D_{a_5,a_1}D_{a_3,a_5}U_{a_4,a_6}\\
&\qquad+D_{a_5,a_1}D_{a_3,a_5}C_{a_4,a_6}
\big)\\
&+\big(
U_{a_5,a_1}U_{a_3,a_5}U_{a_4,a_6}
+U_{a_5,a_1}U_{a_3,a_5}D_{a_4,a_6}
+U_{a_5,a_1}C_{a_3,a_5}C_{a_4,a_6}\\
&\qquad+U_{a_5,a_1}C_{a_3,a_5}D_{a_4,a_6}
+U_{a_5,a_1}D_{a_3,a_5}U_{a_4,a_6}
+U_{a_5,a_1}D_{a_3,a_5}C_{a_4,a_6}\\
&\qquad+C_{a_5,a_1}U_{a_3,a_5}C_{a_4,a_6}
+C_{a_5,a_1}U_{a_3,a_5}D_{a_4,a_6}
+C_{a_5,a_1}C_{a_3,a_5}U_{a_4,a_6}\\
&\qquad+C_{a_5,a_1}C_{a_3,a_5}D_{a_4,a_6}
+C_{a_5,a_1}D_{a_3,a_5}U_{a_4,a_6}
+C_{a_5,a_1}D_{a_3,a_5}C_{a_4,a_6}\\
&\qquad+D_{a_5,a_1}U_{a_3,a_5}U_{a_4,a_6}
+D_{a_5,a_1}U_{a_3,a_5}C_{a_4,a_6}
+D_{a_5,a_1}C_{a_3,a_5}U_{a_4,a_6}\\
&\qquad+D_{a_5,a_1}C_{a_3,a_5}C_{a_4,a_6}
\big)\Big].
\end{align*}

    We evaluate these expressions for each $\Z_3$-coloring of $K_1$ and $K_2$, respectively, with respect to the given biquandle virtual bracket and form Table~\ref{tab:normalized-comparison} consisting of biquandle counting invariants and biquandle virtual bracket polynomial and matrix invariants of $K_1$ and $K_2$.

    \begin{table}[htbp]
    \centering
    \renewcommand{\arraystretch}{1.6}
    \resizebox{\textwidth}{!}{%
    \begin{tabular}{c cc cc cc}
    \toprule
    & \multicolumn{2}{c}{Biquandle counting invariants}
    & \multicolumn{2}{c}{Biquandle virtual bracket invariants}
    & \multicolumn{2}{c}{Normalized biquandle virtual bracket invariants} \\
    \cmidrule(lr){2-3}
    \cmidrule(lr){4-5}
    \cmidrule(lr){6-7}
    Diagram
    & $\Phi_{\mathbb{Z}_3}(K)$
    & $\mathcal{M}_{\mathbb{Z}_3}(K)$
    & $\Phi_{\mathbb{Z}_3}^{\beta}(K)$
    & $\mathcal{M}_{\mathbb{Z}_3}^{\beta}(K)$
    & $\widetilde{\Phi}_{\mathbb{Z}_3}^{\beta}(K)$
    & $\widetilde{\mathcal{M}}_{\mathbb{Z}_3}^{\beta}(K)$ \\
    \midrule
    $K_1$
    & $3$
    & $\begin{pmatrix}
    1&0&0\\
    0&1&0\\
    0&0&1
    \end{pmatrix}$
    & $3u^3$
    & $\begin{pmatrix}
    u^3&0&0\\
    0&u^3&0\\
    0&0&u^3
    \end{pmatrix}$
    & $u+u^4+u^7$
    & $\begin{pmatrix}
    u^7&0&0\\
    0&u&0\\
    0&0&u^4
    \end{pmatrix}$ \\
    \addlinespace[1em]
    $K_2$
    & $3$
    & $\begin{pmatrix}
    1&0&0\\
    0&1&0\\
    0&0&1
    \end{pmatrix}$
    & $3u^3$
    & $\begin{pmatrix}
    u^3&0&0\\
    0&u^3&0\\
    0&0&u^3
    \end{pmatrix}$
    & $3u$
    & $\begin{pmatrix}
    u&0&0\\
    0&u&0\\
    0&0&u
    \end{pmatrix}$ \\
    \bottomrule
    \end{tabular}%
    }
    \caption{Comparison of the biquandle counting invariants,
    biquandle virtual bracket invariants, and their normalized
    versions for $K_1$ and $K_2$.}
    \label{tab:normalized-comparison}
    \end{table}

    It follows that the biquandle counting invariants and the unnormalized biquandle virtual bracket polynomial and matrix invariants of $K_1$ and $K_2$ are the same. However, their normalized biquandle virtual bracket invariants differ. In particular,
    \[
    \widetilde{\Phi}_{\mathbb{Z}_3}^{\beta}(K_1)=u+u^4+u^7
    \neq
    3u=\widetilde{\Phi}_{\mathbb{Z}_3}^{\beta}(K_2),
    \]
    and
    \[\widetilde{\mathcal{M}}_{\mathbb{Z}_3}^{\beta}(K_1)=\begin{pmatrix}
    u^7&0&0\\
    0&u&0\\
    0&0&u^4
    \end{pmatrix}\neq \begin{pmatrix}
    u&0&0\\
    0&u&0\\
    0&0&u \end{pmatrix}=\widetilde{\mathcal{M}}_{\mathbb{Z}_3}^{\beta}(K_2).\]
    Consequently, the normalized biquandle virtual bracket invariants distinguish $K_1$ and $K_2$, while their unnormalized versions do not. This example demonstrates that normalization can yield a strictly stronger invariant when the loop value $\delta$ is a zero divisor in the coefficient ring.
    \end{example}

    \begin{remark}
        The same phenomenon also occurs for classical and virtual links when the loop value $\delta$ of a biquandle virtual bracket is a zero divisor in the coefficient ring. We omit an explicit example, as our focus in this paper is on virtual knotoids.
    \end{remark}

    We next investigate the behavior of the normalized biquandle virtual bracket matrix under connected sum. We first define a matrix product induced by the $\star$ product of polynomials.

    \begin{definition}
    \label{def:bvb-matrix-product}
        Let $\mathcal{A}=(a_{ij})$ and $\mathcal{B}=(b_{ij})$ be $n\times n$ matrices whose entries are polynomials encoding finite
        multisets of elements of $R$. We define $\mathcal{A}\odot\mathcal{B}$ by
        \[
        (\mathcal{A}\odot\mathcal{B})_{ik}
        =
        \sum_{j=1}^{n}
        a_{ij}\star b_{jk}.
        \]
    \end{definition}

    \begin{proposition}
        The product $\odot$ is associative.
    \end{proposition}

    \begin{proof}
        It is straightforward to verify that the operation $\star$ is associative.
        Hence, it follows that for matrices $\mathcal{A}$, $\mathcal{B}$ and $\mathcal{C}$,
        \[
        ((\mathcal A\odot\mathcal B)\odot\mathcal C)_{i\ell}=
        \sum_{j,k} a_{ij}\star b_{jk}\star c_{k\ell}=
        (\mathcal A\odot(\mathcal B\odot\mathcal C))_{i\ell}.
        \]
    \end{proof}

    \begin{theorem}
        Let $X=\{x_1,\ldots,x_n\}$ be a finite biquandle, $\beta$ a
        biquandle virtual bracket over a commutative unitary number ring $R$, and let $K_1$
        and $K_2$ be two virtual knotoid diagrams. Then
        \[
        \widetilde{\mathcal{M}}_X^{\beta}(K_1\#K_2)
        =
        \widetilde{\mathcal{M}}_X^{\beta}(K_1)
        \odot
        \widetilde{\mathcal{M}}_X^{\beta}(K_2).
        \]
    \end{theorem}

    \begin{proof}
        Let $i,k\in\{1,\ldots,n\}$ be fixed. By
        Theorem~\ref{thm:reduced-bvb-multiset-product}, we have
        \[
        \widetilde{\Phi}_{X_{ik}}^{\beta,M}(K_1\#K_2)
        =
        \biguplus_{j=1}^{n}
        \left(
        \widetilde{\Phi}_{X_{ij}}^{\beta,M}(K_1)
        \cdot
        \widetilde{\Phi}_{X_{jk}}^{\beta,M}(K_2)
        \right).
        \]
        The polynomial encoding multiset union is obtained by ordinary polynomial
        addition, while the polynomial encoding multiset product is obtained using
        $\star$. Hence,
        \[
        \widetilde{\Phi}_{X_{ik}}^{\beta}(K_1\#K_2)=
            \sum_{j=1}^{n}\left(
            \widetilde{\Phi}_{X_{ij}}^{\beta}(K_1)\star
            \widetilde{\Phi}_{X_{ij}}^{\beta}(K_2)\right).
        \]
        By Definition~\ref{def:bvb-matrix-product}, we have
        \[
        \widetilde{\mathcal{M}}_X^{\beta}(K_1\#K_2)_{i,k}=\left(\widetilde{\mathcal{M}}_X^{\beta}(K_1)
        \odot
        \widetilde{\mathcal{M}}_X^{\beta}(K_2)\right)_{i,k}
        \]
        Since this holds for every $i,k\in\{1,\ldots,n\}$, the result follows.
    \end{proof}

    \begin{remark}
        The preceding matrix formula is stated for the normalized biquandle virtual
        bracket because it is multiplicative without any assumption on $\delta$. For the unnormalized biquandle virtual bracket matrix, the relation $\bvb{K_f}=\delta\,\bvbn{K_f}$
        together with the connected-sum formula gives
        \[
        \mathcal M_X^\beta(K_1)
        \odot
        \mathcal M_X^\beta(K_2)
        =
        \rho_\delta\!\left(
        \mathcal M_X^\beta(K_1\#K_2)
        \right),
        \]
        where $\rho_\delta$ is defined by
        \[
        \rho_\delta\left(\sum_a m_a u^a\right)=
        \sum_a m_a u^{\delta a},
        \]
        and applied to a matrix entrywise. If $\delta$ is invertible, then
        \[
        \mathcal M_X^\beta(K_1\#K_2)=\rho_{\delta^{-1}}\left(
        \mathcal M_X^\beta(K_1)\odot\mathcal M_X^\beta(K_2)\right).
        \]
        Thus, the normalized formulation of the state-sum gives the connected-sum composition law
        without requiring $\delta$ to be invertible. The corresponding formula for the ordinary biquandle virtual bracket matrix is obtained
        when $\delta$ is invertible.
    \end{remark}

\section{Discussion}
\label{sec:discussion}

    In this paper, we introduced a diagrammatic description of the connected sum of virtual knotoids through virtual connection arcs which can be applied to any pair of representative virtual knotoid diagrams. This description is specific to the virtual setting and particularly convenient for the main purpose of the rest of the paper. Since it is applicable to any pair of given planar virtual knotoid diagrams, it is also useful for computations.

    We then studied the behavior of biquandle-based invariants of virtual knotoids under connected sum. In particular, we obtained explicit formulas relating the invariant of the connected sum in terms of the corresponding invariants of its summands in each case. Besides describing the behavior of these invariants under connected sum, such formulas can substantially simplify computations, particularly for biquandle virtual bracket invariants.

    We also investigated normalized biquandle virtual bracket invariants and their distinguishing power compared to their unnormalized versions. We showed that normalization can yield strictly stronger invariants when the loop value $\delta$ is a zero divisor in the coefficient ring, by giving an example of a pair of virtual knotoids that are distinguished by the normalized invariants but not by the unnormalized ones.

    It would be natural to investigate analogous composition formulas for other enhancements of biquandle coloring invariants of virtual knotoids as a future work.

    \section*{Acknowledgements}
    The authors used ChatGPT for literature search, language polishing, and fixing typos or notational errors. The authors review and take full responsibility for the final content of this manuscript.
    
\bibliographystyle{plain}
\bibliography{references}	

@preamble{
     "\def\cprime{$'$} "
}

@article{carter2002stable,
 author = {Carter, J. Scott and Kamada, Seiichi and Saito, Masahico},
 title = {Stable equivalence of knots on surfaces and virtual knot cobordisms},
 fjournal = {Journal of Knot Theory and its Ramifications},
 journal = {J. Knot Theory Ramifications},
 issn = {0218-2165},
 volume = {11},
 number = {3},
 pages = {311--322},
 year = {2002},
 language = {English},
 doi = {10.1142/S0218216502001639},
 zbMATH = {1790788},
 Zbl = {1004.57007}
}

@article{dorier2018tool,
  title={Knoto-ID: a tool to study the entanglement of open protein chains using the concept of knotoids},
  author={Dorier, Julien and Goundaroulis, Dimos and Benedetti, Fabrizio and Stasiak, Andrzej},
  journal={Bioinformatics},
  volume={34},
  number={19},
  pages={3402--3404},
  year={2018},
  publisher={Oxford University Press}
}

@article{goundaroulis2017topological,
AUTHOR = {Goundaroulis, Dimos and Gügümcü, Neslihan and Lambropoulou, Sofia and Dorier, Julien and Stasiak, Andrzej and Kauffman, Louis},
TITLE = {Topological Models for Open-Knotted Protein Chains Using the Concepts of Knotoids and Bonded Knotoids},
JOURNAL = {Polymers},
VOLUME = {9},
YEAR = {2017},
NUMBER = {9},
ARTICLE-NUMBER = {444},
URL = {https://www.mdpi.com/2073-4360/9/9/444},
PubMedID = {30965745},
ISSN = {2073-4360},
DOI = {10.3390/polym9090444}
}

@article{cattabriga2020,
  author  = {Cattabriga, A. and Horvat, E.},
  title   = {Knot Quandle Decompositions},
  journal = {Mediterranean Journal of Mathematics},
  volume  = {17},
  number  = {3},
  pages   = {98},
  year    = {2020},
  doi     = {10.1007/s00009-020-01530-6}
}

@mastersthesis{pflume,
  author  = {Pflume, Runa},
  title   = {Generalizations of Quandles to Multi-Linkoids},
  school  = {Georg-August-Universit{\"a}t G{\"o}ttingen},
  address = {G{\"o}ttingen, Germany},
  year    = {2023},
  doi     = {10.53846/goediss-10442}
}

@article{gugumcu2026pointed,
author = {G{\"u}g{\"u}mc{\"u}, Neslihan and Pflume, Runa},
title = {Pointed quandles of linkoids},
journal = {Journal of Knot Theory and Its Ramifications},
volume = {35},
number = {13},
pages = {2650030},
year = {2026},
doi = {10.1142/S0218216526500306},

URL = {    https://doi.org/10.1142/S0218216526500306
},
eprint = {https://doi.org/10.1142/S0218216526500306
}
}

@article{cho2019quandle,
	title={Quandle coloring quivers},
	author={Cho, Karina and Nelson, Sam},
	journal={Journal of Knot Theory and Its Ramifications},
	volume={28},
	number={01},
	pages={1950001},
	year={2019},
	publisher={World Scientific}
}

@article{goundaroulis2017studies,
	author  = {Goundaroulis, Dimos and Dorier, Julien and
	Benedetti, Fabrizio and Stasiak, Andrzej},
	title   = {Studies of Global and Local Entanglements of Individual
	Protein Chains Using the Concept of Knotoids},
	journal = {Scientific Reports},
	volume  = {7},
	number  = {1},
	pages   = {6309},
	year    = {2017}
}

@article{gugumcu2026biquandle,
	title={Biquandle virtual brackets and virtual knotoids},
	author={G{\"u}g{\"u}mc{\"u}, Neslihan and Kayaslan, Hamdi},
	journal={Turkish Journal of Mathematics},
	volume={50},
	number={4},
	pages={651--674},
	year={2026},
	publisher={The Scientific and Technological Research Council of T{\"u}rkiye (T{\"U}B{\.I}TAK)}
}

@article{gugumcu2017new,
	title={New invariants of knotoids},
	author={G{\"u}g{\"u}mc{\"u}, Neslihan and Kauffman, Louis H},
	journal={European Journal of Combinatorics},
	volume={65},
	pages={186--229},
	year={2017},
	publisher={Elsevier}
}

@article{kamada2000abstract,
	title={Abstract link diagrams and virtual knots},
	author={Kamada, Naoko and Kamada, Seiichi},
	journal={Journal of Knot Theory and its Ramifications},
	volume={9},
	number={01},
	pages={93--106},
	year={2000},
	publisher={World Scientific}
}

@article{kauffman1999virtual,
	author = {Kauffman, Louis H.},
	title = {Virtual knot theory},
	fjournal = {European Journal of Combinatorics},
	journal = {Eur. J. Comb.},
	issn = {0195-6698},
	volume = {20},
	number = {7},
	pages = {663--690},
	year = {1999},
	language = {English},
	doi = {10.1006/eujc.1999.0314},
	zbMATH = {1382658},
	Zbl = {0938.57006}
}

@article{turaev2012knotoids,
	author = {Turaev, Vladimir},
	title = {Knotoids},
	fjournal = {Osaka Journal of Mathematics},
	journal = {Osaka J. Math.},
	issn = {0030-6126},
	volume = {49},
	number = {1},
	pages = {195--223},
	year = {2012},
	language = {English},
	zbMATH = {6029033},
	Zbl = {1271.57030}
}

@article{kauffman1987state,
	author = {Kauffman, Louis H.},
	title = {State models and the {Jones} polynomial},
	fjournal = {Topology},
	journal = {Topology},
	issn = {0040-9383},
	volume = {26},
	pages = {395--407},
	year = {1987},
	language = {English},
	doi = {10.1016/0040-9383(87)90009-7},
	zbMATH = {4009331},
	Zbl = {0622.57004}
}

@article{gugumcu2025virtual,
	title={Virtual Knotoids in Thickened Surfaces},
	author={G{\"u}g{\"u}mc{\"u}, Neslihan and Kayaslan, Hamdi},
	journal={arXiv preprint arXiv:2502.18160},
	year={2025}
}

@article{gugumcu2021biquandle,
	title={Biquandle brackets and knotoids},
	author={G{\"u}g{\"u}mc{\"u}, Neslihan and Nelson, Sam and Oyamaguchi, Natsumi},
	journal={Journal of Knot Theory and Its Ramifications},
	volume={30},
	number={09},
	pages={2150064},
	year={2021},
	publisher={World Scientific}
}

@article{nelson2017quantum,
	title={Quantum enhancements and biquandle brackets},
	author={Nelson, Sam and Orrison, Michael E and Rivera, Veronica},
	journal={Journal of Knot Theory and Its Ramifications},
	volume={26},
	number={05},
	pages={1750034},
	year={2017},
	publisher={World Scientific}
}

@article{nelson2019biquandle,
	author = {Nelson, Sam and Oshiro, Kanako and Shimizu, Ayaka and Yaguchi, Yoshiro},
	title = {Biquandle virtual brackets},
	fjournal = {Journal of Knot Theory and its Ramifications},
	journal = {J. Knot Theory Ramifications},
	issn = {0218-2165},
	volume = {28},
	number = {11},
	pages = {22},
	note = {Id/No 1940003},
	year = {2019},
	language = {English},
	doi = {10.1142/S0218216519400030},
	zbMATH = {7139440},
	Zbl = {1436.57014}
}
\end{document}